\documentclass{amsart}
\usepackage[latin2]{inputenc}
\usepackage{amsmath}
\usepackage{amssymb}
\usepackage{amsthm}
\usepackage{amscd}
\usepackage{enumerate}
\usepackage{graphicx}

\usepackage{t1enc}
\usepackage{hyperref}
\usepackage{tikz}
\usetikzlibrary{matrix,arrows}
\newcommand{\bigslant}[2]{{\raisebox{.3em}{$#1$}\left/\raisebox{-.3em}{$#2$}\right.}}

\theoremstyle{plain}
\newtheorem{thm}{Theorem}[section]

\newtheorem{cor}[thm]{Corollary}
\newtheorem{lem}[thm]{Lemma}

\newtheorem{rem}[thm]{Remark}
\newtheorem{step}{Step}
\newtheorem{thm*}{Theorem}[]

\theoremstyle{definition}
\newtheorem{defi}[thm]{Definition}

\newtheorem*{exa}{Example}
\newtheorem*{ack}{Acknowledgements}

\theoremstyle{step}
\newtheorem{stepp}{Step}

\date{}
\author{Kasia Jankiewicz}

\address{Dept. of Math. and Stat., McGill University, Montreal, Quebec, Canada}
\email{kasia.jankiewicz@mail.mcgill.ca}
\title{The Fundamental Theorem of Cubical Small Cancellation Theory}
\begin{document}

\maketitle\begin{abstract}We give a new proof of the main theorem in the theory of $\mathrm C(6)$ small cancellation complexes. We prove the fundamental theorem of cubical small cancellation theory for $\mathrm C(9)$ cubical small cancellation complexes.\end{abstract}

\section*{Introduction}
\addcontentsline{toc}{chapter}{Introduction}

Small cancellation theory studies groups with the property that 
the relators in their group presentation have small overlaps with 
each other. The theory, initiated by Tartakovskii \cite{tartakovskii}, 
was developed by Greendlinger and others in the 60s, however 
some ideas appeared much earlier, in the work of Dehn, among others. 
The geometric approach in the study of small cancellation groups, 
i.e.\ the use of disc diagrams, was introduced by Lyndon and can 
be found in Chapter V of \cite{lyndon}. In geometric language, a 
combinatorial $2$-complex satisfies the metric small cancellation 
condition C$'(\frac 1 p)$, if each \emph{piece}, i.e.\ a path arising in 
two ways as a subpath of $2$-cell attaching maps, has length less 
than $\frac 1 p$ of the length of the boundary path of a $2$-cell 
containing the piece. The non-metric small cancellation condition 
$\mathrm C(p)$ requires that the boundary path of each $2$-cell 
cannot be expressed as a concatenation of fewer than $p$ pieces. 
Note that the condition C$'(\frac 1 p)$ implies the condition 
$\mathrm C(p+1)$. The fundamental theorem in the theory takes 
the following form:
\begin{thm*}\label{0.1} Let $X$ be a $\mathrm C(6)$-complex and 
$D\to X$ a minimal disc diagram. One of the following holds:
\begin{itemize}
\item $D$ is a single cell,
\item $D$ is a ladder, or
\item $D$ has at least three spurs or shells of degree $\leq 3$.
\end{itemize}
\begin{figure}\centering
\includegraphics{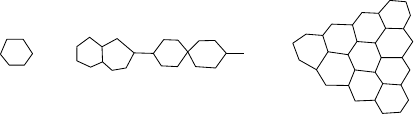}
\caption{Disc diagrams in $\mathrm C(6)$-complexes.}
\end{figure}

\end{thm*}

\noindent For details, see Theorem~\ref{classical}, all definitions 
can be found in Section 1. Theorem \ref{0.1} incorporates a variation 
of Greendlinger's Lemma as well as a "ladder result" that is a variant 
of a classical result on annular diagrams. In its initial form, 
Greendliger's Lemma was used to prove that Dehn's algorithm 
solves the word problem in C$'(\frac 1 6)$ small cancellation groups 
(see \cite{greendlinger}). The above formulation was given by 
McCammond and Wise in \cite{fans}. Unlike the proof presented 
in this paper, their proof uses combinatorial curvature and the 
combinatorial Gauss-Bonnet formula.

Cubical small cancellation theory is a generalization of classical theory 
and it was introduced and developed by Wise in \cite{hierarchy}. 
This builds upon the theory of non-positively curved cube complexes. 
Instead of a standard group presentation, we use cubical presentation, 
i.e.\  express a group as 
$\pi_1(X)/\langle\!\langle \{\phi_{i*}\big(\pi_1(Y_i)\big)\}\rangle\!\rangle$, 
where $X$ is a non-positively curved cube complex and each 
$\phi_i:Y_i\to X$ is a local isometry of cube complexes.Immersed 
complexes $Y_i$ play the role of relators in classical theory. 
We now introduce two types of pieces, contained in the intersection 
either of two ``relators'', or of a ``relator'' and the carrier of a hyperplane. 
Our main result is the following: 
\begin{thm*}\label{main} Let $\langle X,\{Y_i\}\rangle$ satisfy 
$\mathrm C(9)$ and let $(D,\partial D)\to (X^*,X)$ be a minimal 
disc diagram. Then one of the following holds:
\begin{itemize}
\item $D$ is a single vertex or single cone-cell,
\item $D$ is a ladder,
\item $D$ has at least three shells of degree $\leq 4$ and/or corners 
and/or spurs.
\end{itemize}
\end{thm*} 
\noindent The notation is explained in Section~\ref{section cubical presentation}. 
The theorem in the case of $\mathrm C(12)$ is due to Wise and can be 
found in \cite{hierarchy}. Our result partially answers the question on the 
limits of the theory posed by Wise in section 3.r in \cite {hierarchy}. Compared 
to the proof in \cite{hierarchy} our explanation is shorter, self-contained and 
works for the more general condition $\mathrm C(9)$ instead of $\mathrm C(12)$. 
Wise's approach generalizes the classical case in ways we have not engaged 
with, but the most important result there is covered here.

The paper is divided into five sections. Section 1 presents some preliminaries; 
we set up notation and terminology that is used throughout the paper. It also 
provides an exposition of classical small cancellation theory. In Section 2 we 
give a new proof of Theorem \ref{0.1}. In Section 3 we will look more closely 
at non-positively curved cube complexes and prove the following theorem:
\begin{thm*}\label{0.3} Let $X$ be a non-positively curved cube complex and 
$D\to X$ be a minimal disc diagram. Then $D$ is a path graph or it has at 
least three corners and/or spurs.\end{thm*}
This proof is intended to motivate our approach in the proof of Theorem~\ref{main}. 
Section 4 provides the exposition of cubical small cancellation theory. 
Finally, Section 5 is devoted to our main result - Theorem~\ref{main}. 
We first introduce the notion of $D$-walls, which are the crucial tool in 
our approach, and then after a few lemmas we proceed with the proof.

\begin{ack} I am deeply grateful to my advisor, Piotr Przytycki, for being 
simply the best. I would also like to thank Damian Osajda and Daniel Wise 
for helpful discussions and Maciej Zdanowicz for his support.
\end{ack}
\nocite{fans}

\section{Basic definitions}
In this section we give definitions of classical small cancellation theory, 
following mainly \cite{lyndon} and \cite{cubulatingsmall}.

\subsection{Cell complexes}
A map $\phi:X\to Y$ between CW-complexes $X,Y$ is called \emph{combinatorial}, 
if its restriction to any open cell of $X$ is a homeomorphism onto an open cell of $Y$. 
A CW-complex $X$ is called \emph{combinatorial}, if the attaching map of each open 
cell in $X$ is combinatorial for some subdivision of the sphere. We will refer to a closed 
cell as a \emph{cell}. A cell of dimension $0$ is called a \emph{vertex} and a cell of 
dimension $1$ is called an \emph{edge}. Combinatorial map $\phi:X\to Y$ between 
combinatorial complexes $X,Y$ is a \emph{combinatorial immersion} if it is locally injective.

An \emph{$n$-cube} is a copy of $[-1,1]^n$. A \emph{face} of a cube is a subspace 
obtained by restricting some coordinates to $\pm 1$, faces are cubes of lower 
dimension. A \emph{cube complex} is a combinatorial complex whose cells are cubes 
(with subdivision of the boundary consisting of all faces of the cube). A cube of dimension 
$2$ is called a \emph{square}.  

A \emph{valence} of a vertex $v\in X$ is the number of edges in $X$ incident to $v$ with 
loops counted twice. A \emph{path graph} is a $1$-complex $P$ which is homeomorphic 
to an interval (possibly degenerated, i.e.\  a single point). The value $\#\{vertices\}-1$ is 
called the \emph{length} of $P$ and is denoted by $l(P)$. A combinatorial immersion 
$P\to X$ where $P$ is a path graph is called a \emph{combinatorial path}. The images 
of vertices of valence $1$ in $P$ are called \emph{endpoints} of $P$. A path graph of 
length $n$ is denoted by $I_n$. 

\subsection{Disc diagrams}

A \emph{disc diagram} $D$ is a compact, contractible 2-complex with a fixed embedding 
in the plane. A \emph{disc diagram $D$ in $X$} is a combinatorial map $D\to X$ where 
$D$ is a disc diagram. The \emph{boundary path} of $D$ is the attaching map of the 
$2$-cell containing the point at $\infty$ (regarding $S^2=\mathbb{R}^2\cup \infty$). 
See Figure~\ref{boundarypaths}.
\begin{figure}[h]\centering\includegraphics{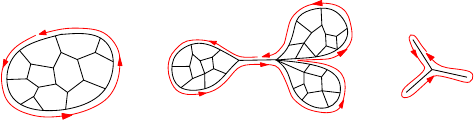}\caption{Boundary paths.}\label{boundarypaths}\end{figure}
The number of $2$-cells in $D$ is called the \emph{area} of $D$ and is denoted by 
$\text{area}(D)$. A \emph{minimal disc diagram} is a disc diagram $D\to X$ with the 
boundary path $P\to X$ such that $\text{area}(D)$ is minimal among all disc diagrams 
with boundary path $P$. If a disc diagram is homeomorphic to a disc, it is called \emph{nonsingular}.
A $2$-cell $C$ is a \emph{boundary cell}, if $C\cap\partial D\neq\emptyset$ and $C$ 
is an \emph{internal cell} otherwise. An edge $e$ is a \emph{boundary edge}, if 
$e\subset \partial D$, $e$ is \emph{semi-internal}, if $e\cap \partial D\neq\emptyset$ 
but $e \not\subset \partial D$ and $e$ is \emph{internal} if $e\cap\partial D=\emptyset$. 
A vertex $v$ is a \emph{boundary vertex}, if $v\in \partial D$, and $v$ is an 
\emph{internal vertex} otherwise. 
A combinatorial path $P\to D$ of length $\geq 1$ with endpoints of valence $\geq 3$ 
in $D$ and all other vertices of valence $2$ in $D$ is called an \emph{arc}. Note that 
every arc is embedded except for endpoints possibly. 
We call an arc $P$ in $D$ a \emph{boundary arc} if $P\subset \partial D$, and an 
\emph{internal arc} otherwise.
The \emph{internal subdiagram} of $D$, denoted by $\text{Int}_D$, is the subcomplex 
consisting of all internal $2$-cells and all arcs that intersect $\partial D$ trivially. 
See Figure~\ref{internalsubdiagram}. A disc diagram which is a cube complex is called 
a \emph{squared disc diagram}. It has cells of three types: vertices, edges and squares.

\begin{figure}[h]\centering\includegraphics{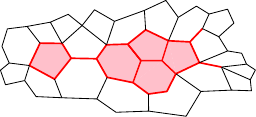}\caption{The internal subdiagram.}\label{internalsubdiagram}\end{figure}

\subsection{$\mathrm C(p)$-small cancellation condition}
Let $X$ be a combinatorial $2$-complex. A nontrivial combinatorial path $P\to X$ is 
a \emph{piece} if there are 2-cells $C_1, C_2$ such that $P\to X$ factors as 
$P\to \partial C_1\to X$ and $P\to\partial C_2\to X$ but there does not exist a 
homeomorphism $\partial C_1\to\partial C_2$ such that the diagram
\begin{center}
\begin{tikzpicture}
[description/.style={fill=white,inner sep=2pt}]
       \matrix (m) [matrix of math nodes, row sep=3em, column sep=2.5em, text height=1.5ex, text depth=0.25ex]
       { P & \partial C_1\\
           \partial C_2& X \\};
       \path[->,font=\scriptsize]
       (m-1-1) edge node[auto] {} (m-1-2)
       (m-1-1) edge node[left] {} (m-2-1)
       (m-1-2) edge node[auto] {} (m-2-2)
       (m-1-2) edge node[auto] {} (m-2-1)
       (m-2-1) edge node[auto] {} (m-2-2);
\end{tikzpicture}
\end{center}
commutes. 
A \emph{maximal piece} is a piece that is not a proper subpath of any piece. 
Note that in a minimal disc diagram $D$ notions of maximal pieces and of internal 
arcs coincide. Every internal and semi-internal edge in such a minimal disc diagram 
is contained in a unique arc, hence in a unique maximal piece. 

Let $p$ be a natural number. A $2$-complex $X$ is \emph{$\mathrm C(p)$-complex} 
(or it satisfies \emph{$\mathrm C(p)$-condition}) if the boundary path of each $2$-cell 
cannot  be expressed as a concatenation of fewer than $p$ pieces in $X$.

\subsection{Spurs and shells}\label{shellsandspurs}
Let $D$ be a disc diagram. A \emph{$k$-shell} of $D$ is a 2-cell $C\to D$ whose 
boundary path $\partial C\to D$ is the concatenation $P_0P_1\cdots P_k$  for some 
$k\geq0$ where $P_0$ is a boundary arc in $D$ and $P_1,\dots, P_k$ are nontrivial 
internal arcs in $D$. The concatenation $P_1\cdots P_k$ is called the \emph{inner path} 
of $C$. The value $k$ is called the \emph{degree} of $C$. A \emph{spur} in $D$ is a 
vertex of valence one in $D$. See Figure~\ref{shells}.

\begin{figure}[h]\centering\includegraphics{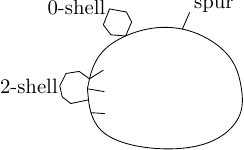}\caption{Spur, $2$-shell and $0$-shell.}\label{shells}\end{figure}

A cell $C$ in $D$ is called a \emph{disconnecting} cell, if $D-C$ is not connected.

\subsection{Ladders}\label{laddersection}
A \emph{ladder} is a disc diagram $L$ consisting of a sequence of 2-cells and/or 
vertices $C_1,C_2,\dots, C_n$ ($n\geq 2$) and edges joining them in the following way:
\begin{itemize}\item if $n=2$ one of the following holds:
\begin{itemize}\item $L=C_1\cup_P C_2$, where $C_1,C_2$ are $2$-cells and 
$P\to C_i$ is an arc for $i=1,2$,
\item $L$ consists of $C_1, C_2$ and an edge $e$ such that $e\cap C_1, e\cap C_2$ 
are two endpoints of $e$,\end{itemize}
\item if $n>2$ for every $1<i<n$ there are exactly two connected components $L', L''$ 
of $L-C_i$ and subdiagrams $L'\cup C_i, L''\cup C_i\subset L$ are both ladders.\end{itemize}
Cells $C_1$ and $C_n$ are called \emph{end-cells}. See Figure~\ref{ladders}.

\begin{figure}[h]\centering\includegraphics{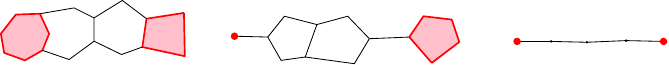}\caption{Examples of ladders. End-cells are marked.}\label{ladders}\end{figure}

\subsection{$n$-Greendlinger Condition}
We say that a disc diagram $D$ satisfies \emph{$n$-Greendlinger Condition} if one of the following holds:\begin{itemize}
\item $D$ is a single cell,
\item $D$ is a ladder, or
\item $D$ has at least three spurs and/or shells of degree $\leq n$, called \emph{exposed cells}.
\end{itemize}

\section{Fundamental theorem of classical small cancellation}
The aim of this section is to give a proof of Theorem \ref{0.1}. First we state and prove 
two lemmas, which will be useful in this proof, as well as later in the case of 
non-positively curved cube complexes and diagrams in cubical small cancellation 
complexes. Then we proceed with the proof of Theorem \ref{0.1}.

\begin{lem}\label{nodisconnect} 
Let $D$ be a disc diagram without disconnecting cells. Then 
either $\text{Int}_D$ is a nontrivial disc diagram, or $D$ consists of at most two cells. 
If $D$ is minimal then so is $\text{Int}_D$.
\end{lem}
\begin{proof} 
The embedding of $\text{Int}_D$ in the plane is induced by the embedding of $D$.
 First, suppose that $D$ has a trivial internal subdiagram. We will show that $D$ has 
 a disconnecting cell or consists of $\leq2$ cells. Suppose $D$ is not a single vertex. 
 Let $b:S^1 \to D$ denote the boundary path of $D$. If $b$ is not an embedding, 
 then either $D$ is a single 1-cell, or there exists vertex $v$ such that $|b^{-1}(v)|>1$. 
 But then $v$ is disconnecting, which contradicts the assumption that $D$ has no 
 disconnecting cells. Suppose $b$ is an embedding. If all the vertices in $\partial D$ 
 have valence two, then $D$ is a single $2$-cell. Suppose that there is a vertex 
 $v\in\partial D$ of valence $\geq 3$ and denote by $P$ an internal arc in $D$ 
 starting at $v$. Since $D$ has a trivial internal subdiagram, the other endpoint of $P$ 
 also lies in $\partial D$. There are two $2$-cells $C_1, C_2$ containing $P$. Observe 
 that $D-P$ is not connected. If there are any $2$-cells in $D$ other than $C_1,C_2$, 
 then one of $C_1,C_2$ is disconnecting, a contradiction. Thus if $D$ has more than 
 two cells then $\text{Int}_D$ is nontrivial.

Now let us prove that in this case $\text{Int}_D$ is compact and contractible. Let 
$H:D\times I\to D$ be a homotopy between $H_0 = id_D$ and a constant map 
$H_1 = p_{x}$ mapping $D$ to $x\in D$ which exists since $D$ is contractible. There is 
a well-defined retraction $r:D\to \text{Int}_D$ mapping each internal arc $P$ such that 
$P\cap\partial D\neq\emptyset$ to its endpoint contained in $\text{Int}_D$ (such an endpoint 
exists since $P$ is not disconnecting) and projecting each boundary $2$-cell $C$ onto 
$C\cap \text{Int}_D$. See Figure~\ref{retraction}.
\begin{figure}[h]\centering\includegraphics{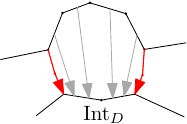}\caption{Retraction $D\to \text{Int}_D$.}\label{retraction}\end{figure}
We have $r\circ \iota = id_{\text{Int}_D}$ where $\iota:\text{Int}_D\to D$ is the inclusion. 
Then $r\circ H\circ \iota:\text{Int}_D\times I \to \text{Int}_D$ is a homotopy between 
$r\circ id_D \circ \iota = id_{\text{Int}_D}$ and a constant map $r \circ p_x \circ \iota = p_{r(x)}$. 
Thus $\text{Int}_D$ is contractible.
Finally, $\text{Int}_D$ is compact since it is the image of the compact space $D$ under the continuous map $r$. 

The minimality of $\text{Int}_D$ assuming the minimality of $D$ is immediate.

\end{proof}

We write $D=D_1\cup_C D_2$ if $D=D_1\cup D_2$ and $C = D_1\cap D_2$. 

\begin{lem}\label{pushout}
Let $D=D_1\cup_C D_2$ be a disc diagram where $D_1,D_2$ are disc diagrams and 
$C$ is a single cell.  If $D_1$ and $D_2$ satisfy $n$-Greendlinger Condition, 
then so does $D$. \end{lem}
\begin{proof} Suppose $C$ is a $1$-cell. If for $i=1,2$ the disc diagram $D_i$ is a single 
$2$-cell and $C\subset D_i$ then $D$ is a ladder. Otherwise either $C$ contains a $0$-cell 
which is disconnecting in $D$ or $C$ is contained in a $2$-cell which is disconnecting in $D$. 
Thus it suffices to consider cases where $C$ is a $0$-cell or a $2$-cell. If one of $D_1, D_2$ 
consists only of $C$, then there is nothing to show, so we assume that $C\subsetneq D_i$ 
for $i=1,2$. Observe that $C$ is a boundary cell in $D_i$ for $i=1,2$, because otherwise 
$D$ would not embed in the plane. If one of $D_1$ and $D_2$ has at least three exposed cells, 
say $D_1$, then at least two exposed cells of $D_1$ remain exposed in $D$, since we glue along 
a single cell. Also, at least one of exposed cells of $D_2$ remains an exposed cell of $D$. 
See Figure~\ref{pushoutfig}. 
\begin{figure}[h]\centering\includegraphics{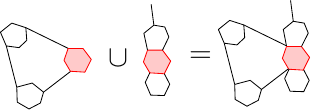}\caption{We glue disc diagram with 
three exposed cells and ladder along the marked cell.}\label{pushoutfig}\end{figure}
If $D_1$ and $D_2$ are both ladders or single cells, then one of the following holds:
\begin{itemize}\item the diagram $D$ is a ladder, if for both $i=1,2$ either $D_i$ is 
a single $2$-cell or $C$ is contained in an end-cell of $D_i$,
\item the diagram $D$ has three exposed cells, otherwise. To see this suppose that 
$D_1$ is a ladder and $C$ is not contained in any end-cell of $D_1$. Then the end-cells 
of $D_1$ remain exposed in $D$ and there is at least one end-cell of $D_2$ which remains 
exposed in $D$.
\end{itemize}
\end{proof}

\begin{thm}\label{classical}Let $X$ be a $\mathrm C(6)$-complex and $D\to X$ a minimal disc diagram. 
Then $D$ satisfies $3$-Greendlinger Condition.\end{thm}
\begin{proof}
We prove the theorem inductively on the number of $2$-cells. It suffices to check 
$3$-Greendlinger Condition for a disc diagram $D$ with no disconnecting cells. 
Indeed, if $D$ has a disconnecting cell, i.e.\  $D=D_1\cup_{C} D_2$, then by the induction assumption 
they both satisfy $3$-Greendlinger Condition and by Lemma $2.2$ so does $D$. If $D$ is a single cell, 
there is nothing to prove. If $D$ consists of two cells, then it is a ladder. From now on we assume 
that $D$ has no disconnecting cells and that it has $\geq3$ cells. By Lemma~\ref{nodisconnect} 
the internal subdiagram $\text{Int}_D$ is a minimal disc diagram with fewer internal cells than $D$ 
and by the induction assumption $\text{Int}_D$ satisfies $3$-Greendlinger Condition.

We now consider different cases depending on what $\text{Int}_D$ is and we show that in every case 
there are at least three exposed cells in $D$. The case of a ladder is the last one.

\begin{itemize}\item ($\text{Int}_D$ is a single vertex $v$) The valence of $v$ in $D$ is at least $3$ 
(if it was $2$, then $v$ would belong to the arc with endpoints in $\partial D$, so $v$ would not 
belong to $\text{Int}_D$), so there are at least three $2$-cells attached to $\text{Int}_D$, 
they are 2-shells in $D$.
\item ($\text{Int}_D$ is a single $2$-cell) There is a $2$-cell attached to each arc in $\text{Int}_D$, 
it is a $3$-shell of $D$. By $\mathrm C(6)$ there are at least six arcs, hence there are at least 
six $3$-shells in $D$.

\item ($\text{Int}_D$ has at least three exposed cells) Let $C$ be an exposed cell of $D$, 
if $C$ is a spur, then there is at least one $2$-shell in $D$ attached to the endpoint of $C$. 
Suppose $C$ is a shell. Let $P_1,\dots P_k$ be arcs in $D$ such that the concatenation 
$P_1\cdots P_k$ is a boundary arc of $C$ in $\text{Int}_D$. Since $C$ is an exposed cell 
in $\text{Int}_D$ and $D$ is $\mathrm C(6)$ diagram, we have $k\geq 3$. The $2$-cell 
attached to $P_2$ is a $3$-shell in $D$. See Figure~\ref{exposedcell}. Hence, 
for any exposed cell $C$ of $\text{Int}_D$ there is an exposed cell of $D$ attached to $C$ 
and to no other cell of $\text{Int}_D$.

\begin{figure}[h]\centering\includegraphics{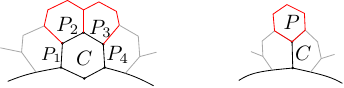}
\caption{The cell $C$ is exposed in the internal subdiagram. The cells $P_2,P_3$ on the left and $P$ on the right are exposed in $D$.}\label{exposedcell}\end{figure}

\item ($\text{Int}_D$ is a ladder) First suppose that $\text{Int}_D$ is a path graph of non-zero length. 
Let $\text{Int}_D=P_1\cdots P_k$ where $P_i$ are arcs in $D$. There is at least one $2$-shell in $D$ 
attached to each endpoint of $\text{Int}_D$. At most one of $2$-shells attached to $P_1$ is also 
attached to $P_2$, so there is one which is a $3$-shell. See Figure~\ref{internaldiagramladder}.
\begin{figure}[h]\centering\includegraphics{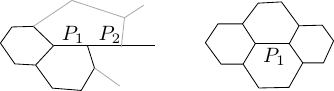}
\caption{The internal subdiagram is a path graph of length $\geq2$ and of length $1$.}\label{internaldiagramladder}\end{figure}

If $\text{Int}_D$ is not a path graph, then there is a $2$-cell in $\text{Int}_D$ which we assume now. 
Suppose one of the end-cells of $\text{Int}_D$, say $C_1$, is a $2$-cell in $\text{Int}_D$. Let 
$P_1,\dots, P_k$ be arcs in $D$ such that $P_1\cdots P_k$ is a boundary arc of $C_1$ in $\text{Int}_D$. 
There is a $3$-shell in $D$ attached to $P_i$ for each $i=2,\dots, k-1$. By $\mathrm C(6)$ 
we have $k\geq 5$, so there are at least three such $3$-shells. If both end-cells are vertices, 
then there is at least one exposed cell attached to each of them. Denote by $i$ the minimal index 
such that $C_i$ from the definition of ladder is a $2$-cell. There is a boundary arc $P=P_1\cdots P_k$ 
of $C_i$ in $\text{Int}_D$ where $P_1\dots, P_k$ are arcs in $D$. As before there is a $3$-shell in $D$ 
attached to $P_i$ for each $i=2,\dots, k-1$. By $\mathrm C(6)$ we have $k\geq 3$, so there are 
at least one such $3$-shell and in total there are at least three $3$-shells in $D$.

\end{itemize}
\end{proof}

\section{Non-positively curved cube complexes}
In this section we give a brief exposition of cube complexes, following \cite{special} or \cite{hierarchy} and prove Theorem \ref{0.3}. The approach applied here will be later adapted in the proof of the main theorem.

\subsection{Non-positively curved cube complexes}
A \emph{link} of a vertex $v$ in a cube complex $X$ is a complex whose vertices correspond 
to oriented edges incident to $v$ and there is an $n$-simplex spanned on a collection of vertices 
whenever there is an $(n+1)$-cube $C\to X$ such that corresponding edges in $X$ are images 
of faces of $C$ containing a vertex $\bar v$ which is mapped to $v$. It can be thought as 
an intersection of a small radius sphere around the vertex $v$ in $X$. A \emph{flag complex} 
is a simplicial complex, such that each set of vertices pairwise connected by edges spans a simplex. 
A cube complex is called \emph{non-positively curved} if all its vertex links are flag. 
A CAT(0) \emph{cube complex} is a simply connected, non-positively curved cube complex. 

Let $X$ be a non-positively curved cube complex and $D\to X$ a disc diagram in $X$ 
consisting of a three squares incident to a vertex $v$  that are pairwise intersecting along one edge. 
Since $X$ is non-positively curved, there exists the disc diagram $D'\to X$ with 
$\partial D=\partial D'$, such that $D\cup_{\partial D=\partial D'}D'$ is the $2$-skeleton of 
a $3$-dimensional cube $Q\to X$. The replacement of $D$ by $D'$ is called a \emph{hexagon move}. 
See Figure~\ref{hexagonmove}. For every square $C\to D$ there exists the unique square 
$\hat C\to D'$ such that $C\cap\hat C=\emptyset$ in $Q$. Such $\hat C$ is called \emph{opposite} to $C$. 

\begin{figure}[h]\centering\includegraphics{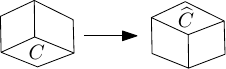}\caption{Hexagon move.}\label{hexagonmove}
\end{figure}

\subsection{Corners}\label{seccorners}
Let $D\to X$ be a disc diagram in a cube complex $X$. A boundary vertex $v$ of valence $2$ 
contained in some square $C$ in $D$ is a \emph{corner}. 
The square $C$ is called a \emph{corner-square}. See Figure~\ref{corner}. Note that 
a corner-square may contain more than one corner.
 
\begin{figure}[h]\centering\includegraphics{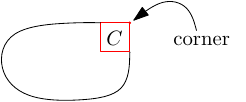}\caption{Corner-square $C$.}\label{corner}
\end{figure}

\subsection{Hyperplanes}\label{sechyperplanes}
Let $X$ be a cube complex. A \emph{midcube} is a subspace of a cube $[-1,1]^n$ obtained 
by restricting one coordinate to $0$. A midcube of an edge is called a \emph{midpoint}. 
Let $H$ be a new cube complex whose cubes are midcubes of $X$ and attaching maps 
are restrictions of attaching maps in $X$ to midcubes. A connected component $\Gamma$ 
of $H$ is called an \emph{immersed hyperplane}. 

There is a natural immersion $\Gamma \to X$ and we will often think of hyperplanes as 
subspaces of $X$. See Figure~\ref{hyperplanes}. An immersed hyperplane $\Gamma$ 
is said to be \emph{dual} to an edge $e$ if a midpoint of $e$ is a vertex of $\Gamma$. 
The immersed hyperplane dual to $e$ is denoted by $\Gamma(e)$. We say that edges 
$e,e'$ are \emph{parallel} if $\Gamma(e)=\Gamma(e')$.
\begin{figure}[h]\centering\includegraphics{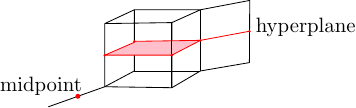}\caption{Hyperplane.}\label{hyperplanes}
\end{figure}

The \emph{carrier} $N(\Gamma)$ of an immersed hyperplane $\Gamma$ in $X$ 
(or of a subcomplex $\Gamma$ of an immersed hyperplane) is a cube complex defined 
as follows: for each cube $C$ in $\Gamma$ we take the copy of the cube in $X$ whose 
midcube is $C$ and two such cubes are attached to each other along faces if corresponding 
midcubes are attached to each other in $\Gamma$ along midcubes of these faces. 
By the construction, we have a map $\iota:N(\Gamma)\to X$. Whenever $\iota$ is an embedding 
we write $N(\Gamma)$ instead of $\iota\big(N(\Gamma)\big)$.

The immersed hyperplanes in a squared disc diagram are immersed path graphs. 
Suppose $\Gamma$ is an immersed hyperplane in a squared disc diagram $D$ such that 
$\iota:N(\Gamma)\to D$ is an embedding. Denote by $K$ one of two connected components 
of $D-\Gamma$. We define the \emph{$\Gamma$-component} corresponding to $K$ as 
$K \cup N(\Gamma)$, i.e.\  this is the minimal subdiagram of $D$ which contains $K$. 
See Figure~\ref{components}.

 \begin{figure}[h]\centering\includegraphics{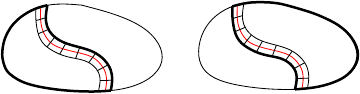}\caption{The $\Gamma$-components.}\label{components}
\end{figure}

Let $D$ be a squared disc diagram with no disconnecting cells and nontrivial internal 
subdiagram. A hyperplane $\Gamma$ is \emph{collaring}, if it is not dual to any internal 
edge in $D$. We say that $D$ is \emph{collared} if $\Gamma(e)$ is collaring for every 
semi-internal edge $e$. We say that $D$ is \emph{collared by $\{\Gamma_1,\dots, \Gamma_n\}$} 
if for every semi-internal $e$ there exists $i$ such that $\Gamma(e)=\Gamma_i$ and all 
$\Gamma_i$ are collaring.  
 \begin{figure}[h]\centering\includegraphics{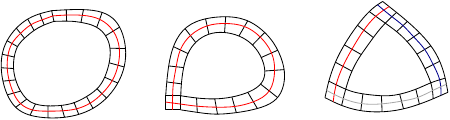}\caption{First two disc diagrams are collared by a single hyperplane and the last disc diagram is bounded by three hyperplanes.}
\end{figure}

\begin{lem}\label{collaredbyhyperplanes}
Let $D$ be a squared disc diagram with no disconnecting cells that has at least one corner. 
The following are equivalent\begin{enumerate}[$(1)$]\item $D$ is collared,
\item every hyperplane dual to an edge containing a corner of $D$ is collaring,
\item all boundary vertices of $D$ have valence $\leq3$.\end{enumerate}\end{lem}
\begin{proof}\item
$(1)\Rightarrow(2)$ This implication is trivial.\\
$(2)\Rightarrow(3)$ If a boundary vertex $v$ has valence $>3$, then no hyperplane dual to 
a semi-internal edge containing $v$ is collaring. Let $P$ be a minimal subpath of $\partial D$ 
with corners of $D$ as endpoints and denote by $v_1, \dots, v_n$ the consecutive vertices of $P$. 
Let $1<i<n$ be the minimal number such that valence of $v_i$ in $D$ is $>3$. Let $e$ be the edge 
in $D$ that contains $v_1$ but is not contained in $P$. The hyperplane $\Gamma(e)$ is not collaring. 
See Figure~\ref{4valence}.
 \begin{figure}[h]\centering\includegraphics{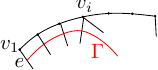}\caption{The hyperplane $\Gamma(e)$ dual to 
 an edge $e$ containing corner $v_1$ is not collaring.}\label{4valence}
\end{figure}\\
$(3)\Rightarrow(1)$ If $D$ is not collared, then there exists an immersed hyperplane dual to 
some internal edge and some semi-internal edge in $D$. There exists a square $C$ in $D$ 
such that one of its edges $e$ is semi-internal and the opposite one $\bar e$ is internal. 
Denote by $v$ the boundary vertex contained in $e$. The valence of $v$ is $\geq4$, 
because otherwise $\bar e$ would contain a boundary vertex.
\end{proof}

\subsection{Disc diagrams in non-positively curved cube complexes}

\begin{thm}\label{cubes}Let $X$ be a non-positively curved cube complex and $D\to X$ a minimal disc diagram. Then $D$ is a path graph or it has at least three corners and/or spurs.\end{thm}

\begin{proof}
\begin{stepp}It suffices to verify $2$-Greendlinger Condition. \end{stepp}
Indeed,
\begin{itemize}
\item if $D$ is a single 0-cell or a ladder consisting only of 1-cells, then $D$ is a path graph,
\item if $D$ is a single square or a ladder with at least one 2-cell or $D$ has at least three 
shells of degree $\leq2$ and/or spurs, then $D$ has at least three corners and/or spurs. 
To see that note that shells of degree $\leq2$ in $D$ are corner-squares.\end{itemize} 

\noindent We show that $2$-Greendlinger Condition is satisfied by induction on the number of cells. 

\begin{stepp}
All cells in $D$ are embedded and that the intersection of two cells consists of exactly one cell. 
\end{stepp}
Since $X$ is non-positively curved, no square in $D$ has two consecutive edges glued, 
because otherwise on of the vertex links in $X$ would contain a loop. Similarly if 
there are two squares with $\geq2$ consecutive common edges in $D$, then since 
$X$ is non-positively curved they are mapped to the same square in $X$, thus $D$ 
is not minimal. See Figure~\ref{twosq}. 
\begin{figure}[h]\centering\includegraphics{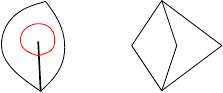}\caption{These cannot be minimal disc diagrams in a non-positively curved cube complex.}\label{twosq}\end{figure}
Suppose that $S$ is a non-embedded square in $D$. Let $P\to\partial S$ be a minimal subpath 
whose endpoints are mapped to the same point in $D$ such that $S$ is not contained in 
the subdiagram $D'$ of $D$ bounded by $P$. See the left diagram in Figure~\ref{notembeddedsquare}. 
\begin{figure}[h]\centering\includegraphics{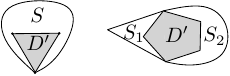}\caption{Subdiagram $D'$.}\label{notembeddedsquare}\end{figure}
By the minimality of $P$ the diagram $D'$ has no spurs. Since two squares cannot have 
two consecutive edges in common, the only possible corner of $D'$ is the endpoint of $P$. 
Thus $D'$ contradicts the induction assumption. Now suppose that there are two cells 
$S_1,S_2$ in $D$ whose intersection is not connected. Let $P_1\to\partial S_1,P_2\to\partial S_2$ 
be minimal subpaths with common endpoints in $D$ such that none of $S_1, S_2$ is 
contained in the subdiagram $D'$ of $D$ bounded the concatenation of $P_1$ and $P_2$. 
See the right diagram in Figure~\ref{notembeddedsquare}. As before there are no spurs in $D'$ 
and there are at most two corners, so $D'$ contradicts the induction assumption. Similarly 
all edges are embedded and every two edges have at most one common vertex.

\begin{stepp}It suffices to verify $2$-Greendlinger Condition for $D\to X$ that has no disconnecting cells.
\end{stepp}
It follows from Lemma~\ref{pushout}. If $D$ is a single cell, there is nothing to prove. If $D$ 
consists of two cells, then it contains a disconnecting cell. Thus we can restrict our attention 
to diagrams with $\geq3$ squares. In such case we need to show that $D$ has $\geq 3$ corner-squares.

\begin{stepp}From now on, we assume that $D$ has no disconnecting cells. The carriers 
of immersed hyperplanes in $D$ embed. \end{stepp} Suppose to the contrary that $\Gamma$ 
is an immersed hyperplane such that $N(\Gamma)$ does not embed. Let $\Gamma'$ be 
a minimal subpath of $\Gamma$ such that $N(\Gamma')$ does not embed. See Figure~\ref{notembedded}.
\begin{figure}[h]\centering\includegraphics{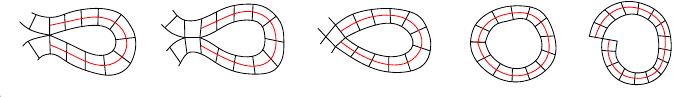}\caption{The hyperplane $\Gamma$ where $N(\Gamma)$ does not embed.}\label{notembedded}\end{figure}
Denote by $D_{\Gamma'}$ the minimal disc diagram that contains $N(\Gamma')$ and 
the \emph{internal} connected component of $D-N(\Gamma')$, i.e.\  the unique component 
which has trivial intersection with $\partial D$. One of the following holds:\begin{itemize}
\item The diagram $D_{\Gamma'}$ is a proper subdiagram of $D$, so $D_{\Gamma'}$ has 
fewer cells than $D$. But $D'$ has at most two corner-squares (images of end-cells of 
$N(\Gamma')$ possibly) and no spurs, hence we get a contradiction with the induction assumption.
\item We have $D_{\Gamma'}=D$. Set $S$ to be any square in $N(\Gamma')$ that is not 
a corner-square and let $\Gamma_S$ be the immersed hyperplane dual to the unique 
boundary edge of $S$. See Figure~\ref{notembeddedcarrier}. 
\begin{figure}[h]\centering\includegraphics{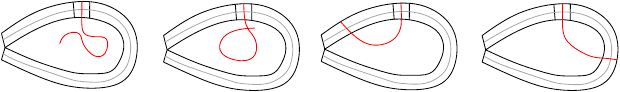}\caption{The diagram $D = D_{\Gamma'}$ and hyperplane $\Gamma_S$.}\label{notembeddedcarrier}\end{figure}
If $N(\Gamma_S)$ is not embedded, then we proceed as in the previous case and get 
a proper subdiagram $D_{\Gamma_S}$ of $D$ with at most two corner squares and no spurs, 
a contradiction. Thus $N(\Gamma_S)$ embeds. One of $\Gamma_S$-components has 
$2$ exposed squares and fewer cells than $D$, a contradiction with the induction assumption. \end{itemize}
Thus carriers of hyperplanes embed in $D$.

\begin{stepp}The diagram $D$ satisfies $2$-Greendlinger condition.\end{stepp}
Suppose, contrary to $2$-Greendlinger Condition, that $D$ has at most two corner-squares. 
We now show that in this case $D$ is collared. If $D$ had no corner-squares, then any 
$\Gamma(e)$-component, for any boundary edge $e$ would have $\leq2$ corner-squares, 
no spurs  and obviously less cells than $D$, which would contradict the induction assumption. 
Thus $D$ has some corner-squares. By Lemma~\ref{collaredbyhyperplanes} it suffices to verify 
that all hyperplanes dual to edges containing corners are collaring. Let $e'$ be such an edge. 
If $\Gamma(e')$ was not collaring then one of $\Gamma(e')$-components would have at most 
two corner-squares, no spurs and fewer cells than $D$, which is impossible by the induction assumption. 
See Figure~\ref{bigonss}. Thus $D$ is collared.

\begin{figure}[h]\centering\includegraphics{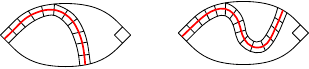}\caption{The hyperplane $\Gamma(e')$. }\label{bigonss}\end{figure}

By Lemma~\ref{nodisconnect} the internal subdiagram $\text{Int}_D$ is a minimal disc diagram, 
so by the induction assumption it satisfies $2$-Greendlinger Condition and therefore it is a path graph, 
or it has at least three corners and/or spurs. Let us consider these two cases separately.
\begin{itemize}
\item ($\text{Int}_D$ has at least three corners and/or spurs) There is a corner-square of $D$ 
attached to each spur of $\text{Int}_D$, since valence in $D$ of a spur of $\text{Int}_D$ is $\geq3$. 
Suppose that there is a corner $v$ in $\text{Int}_D$, let $S$ be a square in $\text{Int}_D$ containing $v$. 
If the valence of $v$ in $D$ is $\geq4$ then there is a corner-square of $D$ containing $v$. 
See Figure~\ref{interiorinsquarediagram}. If the valence is 3, then by a hexagon move 
applied to squares containing $v$ we obtain a disc diagram $D'$ with the same number of cells as $D$ 
and $\partial D=\partial D'$. Since $D$ is collared, by Lemma~\ref{collaredbyhyperplanes} 
all boundary vertices in $D$ have valence $\leq3$, so square $\widehat S$ opposite to $S$ 
is a corner-square of $D'$ and the diagram $\overline{D'- \widehat S}$ has $\leq2$ corners 
and no spurs, a contradiction. See Figure~\ref{interiorinsquarediagram}.

\begin{figure}[h]\centering\includegraphics{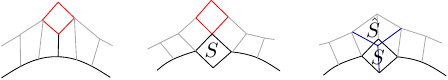}\caption{Internal diagram contains a spur, a corner whose valence in $D$ is $\geq4$, or a corner whose valence in $D$ is $3$.}\label{interiorinsquarediagram}\end{figure}

\item ($\text{Int}_D$ is a path graph) There are corner-squares $S_1,S_2$ each incident to 
one endpoint of $\text{Int}_D$. Since $D$ is collared by Lemma~\ref{collaredbyhyperplanes}, 
the diagram $D$ has no boundary vertices of valence~$>3$. Since $S_1, S_2$ are the only 
corner-squares in $D$, we have $\overline{D-(S_1\cup S_2)}=I_2\times I_n$ where $n\geq1$ 
is the length of $\text{Int}_D$. See the left diagram in Figure~\ref{bigon-after-hexagon-move}. 
Applying a hexagon move to squares containing an endpoint of $\text{Int}_D$ we obtain a diagram 
with a proper subdiagram $D'$ with two corners and no spurs, which is a contradiction and 
completes a proof. See the right diagram in Figure~\ref{bigon-after-hexagon-move}.
\begin{figure}[h]\centering\includegraphics{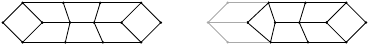}
\caption{On the left, the diagram $D$ where $\text{Int}_D$ is a path graph of length $3$. On the right, the diagram $D'$ obtained from $D$ by a hexagon move.}\label{bigon-after-hexagon-move}\end{figure}

\end{itemize}
\end{proof}

\begin{cor}\label{cubecor}If $X$ is a non-positively curved cube complex and $D\to X$ is a minimal nonsingular disc diagram, then $D$ has at least three corners.\end{cor}

\subsection{Convexity}
Let $X$ be a CAT(0) cube complex. A subcomplex $Y\subset X$ is \emph{convex} if 
for any vertices $v, v'\in Y$ every combinatorial path $P\to X$ of minimal length with endpoints 
$v$ and $v'$ is contained in $Y$. A combinatorial immersion $\phi:Y\to X$ of cube complexes 
is called a \emph{local isometry} provided that for any pair of edges $e,e'$ incident to 
a vertex $v$ in $X$, if the vertices in link of $\phi(v)$ corresponding to $\phi(e),\phi(e')$ are adjacent, 
then the vertices corresponding to $e,e'$ in the link of $v$ are also adjacent. Note that, 
if $X$ is a non-positively curved cube complex and $\Gamma$ is an immersed hyperplane in $X$, 
then $\iota:N(\Gamma)\to X$ is a local isometry.

\begin{lem}\label{convex}Let $X,Y$ be non-positively curved cube complexes and let $\phi:Y\to X$ 
be a local isometry. Then the lift $\widetilde\phi:\widetilde Y\to\widetilde X$ is an embedding and 
the image is a convex subcomplex of $\widetilde X$, where $\widetilde X,\widetilde Y$ are 
universal covers of $X,Y$.\end{lem}

\begin{proof}
Let $v,v'\in \widetilde Y$ be vertices. It suffices to verify that any minimal length combinatorial path 
in $\widetilde X$ joining $\widetilde\phi(v),\widetilde\phi(v')$ is the image under $\widetilde\phi$ of 
a minimal length combinatorial path in $\widetilde Y$ joining $v,v'$. Let $D\to \widetilde X$ be 
a minimal disc diagram with boundary path $\widetilde\phi(\beta)\bar\gamma$, where $\beta$ 
is a minimal length combinatorial path in $\widetilde Y$ joining $v,v'$ and $\gamma$ is 
a minimal length combinatorial path joining $\widetilde\phi(v),\widetilde\phi(v')$ ($\bar \gamma$ 
denotes the path $\gamma$ with reversed direction). By induction on $\text{area}(D)$ over pairs 
$\beta, \gamma$ we show that there exists a minimal length combinatorial path $\alpha$ in 
$\widetilde Y$ such that $\widetilde\phi(\alpha)=\gamma$. If $\text{area}(D)=0$, then 
by Theorem~\ref{cubes}
\begin{itemize}
\item either $D$ is a path graph with endpoints $\widetilde\phi(v),\widetilde\phi(v')$, 
and what follows $\widetilde\phi(\beta)=\gamma$,
\item or there is a spur $w$ in $D$, distinct from $v,v'$. If $w\in \gamma$, 
then the length of $\gamma$ is not minimal. Otherwise, if $w\in \widetilde\phi(\beta)$, 
then since $\widetilde\phi$ is a combinatorial immersion, the length of $\beta$ is not minimal.
\end{itemize}
Now suppose that $\text{area}(D)=n>0$. By Theorem~\ref{cubes} there are at least three corners 
and/or spurs in $D$, so at least one corner or spur distinct from $v,v'$, let us denote it by $w$. 
If $w$ is a spur, we conclude as before, that the length of one of $\gamma,\beta$ is not minimal. 
Thus $w$ is a corner, denote by $S$ the square in $\widetilde X$ containing $w$. 
If $w\in\widetilde\phi(\beta)$, then since $\widetilde \phi$ is a local isometry, there is 
a square $S'$ in $\widetilde Y$ such that $\widetilde\phi(S')=S$. Let $e_1e_2e_3e_4=\partial S'$ 
such that $\phi(e_1),\phi(e_2)$ both contain $w$ and their concatenation $e_1e_2$ is a subpath 
of $\beta$. Let $\beta'$ be the path obtained from $\beta$ by replacing $e_1e_2$ by 
$\bar{e_4}\bar{e_3}$ ($\bar e$ denotes the edge $e$ with reversed direction). 
See Figure~\ref{newpath}.
\begin{figure}[h]\centering\includegraphics{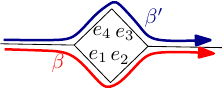}\caption{Replacing $e_1e_2$ by $\bar{e_4}\bar{e_3}$.}\label{newpath}\end{figure}
The path $\beta'$ is joining $v,v'$ in $\widetilde Y$ and the minimal area of disc diagram 
with the boundary path $\widetilde\phi(\beta')\bar\gamma$ is equal to $n-1$. 
By induction assumption there is a minimal length combinatorial path $\alpha$ in $\widetilde Y$ 
such that $\widetilde\phi(\alpha)=\gamma$. 

Now suppose $w\in\gamma$ and let $e_1e_2e_3e_4=\partial S$ such that $e_1,e_2$ 
both contain $w$ and their concatenation $e_1e_2$ is a subpath of $\gamma$. 
Let $\gamma'$ be the path obtained from $\gamma$ by replacing subpath $e_1e_2$ by 
$\bar{e_4}\bar{e_3}$. The  minimal area of a disc diagram with the boundary path 
$\widetilde\phi(\beta)\bar\gamma'$ is $n-1$, thus by the induction assumption there is 
a minimal length combinatorial path $\alpha'$ in $\widetilde Y$ such that 
$\widetilde\phi(\alpha')=\gamma'$. Since $\widetilde\phi$ is a local isometry, 
there exists a minimal length combinatorial path $\alpha$ in $\widetilde Y$ such that 
$\widetilde\phi(\alpha)=\gamma$.

\end{proof}

\begin{cor}\label{convexcarrier}The carrier $N(\Gamma)$ of a hyperplane $\Gamma$ in a \emph{CAT(0)} cube complex $X$ is a convex subcomplex.\end{cor}

\section{Cubical small cancellation theory}
The aim of this section is to describe cubical small cancellation theory, due to Wise, 
which is a generalization of classical small cancellation theory. In the beginning we define 
cubical presentations introduced by Wise, following \cite{mixed}, \cite{raags} or \cite{hierarchy}. 
Secondly, we introduce the notion of pseudorectangles and use it to define ladders. 
This definition is consistent with the one given in \cite{hierarchy}, but is more general 
than the one in \cite{raags} and \cite{mixed}, which is equivalent to our definition with 
the restriction that the joining pseudorectangles are actual rectangles. Then, we define 
cone-pieces and hyperplane-pieces, there are several equivalent definitions of pieces 
in cubical presentation complexes, here we  follow \cite{mixed}. This allows us to formulate 
cubical small cancellation conditions. Notice, that we use different terminology than in 
\cite{mixed}, i.e. our hyperplanes-pieces are referred to as \emph{wall-pieces} there. 
Finally, we introduce $D$-pieces for a disc diagram $D$ in a cube complex $X$, which 
are aimed to correspond to pieces in $X$.

\subsection{Cubical presentation}\label{section cubical presentation}
A \emph{cubical presentation} $\langle X,\{Y_i\}\rangle$ consists of a non-positively curved 
cube complex $X$ and a family of local isometries of cube complexes $\phi_i:Y_i\to X$. 
The group $G$ assigned to a cubical presentation is the quotient
\[G=\pi_1(X)/\langle\!\langle \{\phi_{i*}\big(\pi_1(Y_i)\big)\}\rangle\!\rangle\]
Let
\[X^*=X\cup \bigcup_i \text{Cone}(Y_i)/\{(y_i,0)\sim \phi_i(y_i)\text{ for all }y_i\in Y_i\},\]
where $\text{Cone}(Y)=Y\times[0,1]/Y\times\{1\}$. Then we have $G=\pi_1(X^*)$. 
We regard $X^*$ as a cell complex with cells divided into two families: cubes and pyramids 
(i.e.\ cones on single cubes). We will refer to $X^*$ as a \emph{presentation complex}. 
There is a natural combinatorial inclusion $X\to X^*$. The vertices of $X^*$ which 
are not contained in $X$ are called \emph{cone-points}. By van Kampen lemma 
(see \cite{vankampen}), for every closed combinatorial path $P\to X$ such that 
the composition $P\to X\to X^*$ is null-homotopic, there exist a disc diagram 
$(D,\partial D)\to (X^*,X)$ with boundary path $P$. The $2$-cells of $D$ are either 
squares of $X$ or triangles (i.e.\  cones on edges) in the cone $\text{Cone}(Y_i)$ 
for some $i$. The points which are mapped to  cone-points in $X^*$ are also called 
\emph{cone-points} of $D$. Triangles in $D$ are grouped together into cyclic families 
meeting around a cone-point $v$, 
such families form polygons which we call \emph{cone-cells}. From now on we regard 
$D$ as a cell complex with $2$-cells divided into two families: squares and cone-cells. 
We define the \emph{complexity} of a disc diagram $D$ as the following
\[\text{Comp}(D)=(\#\text{cone-cells}, \#\text{squares}).\]
The disc diagram $(D,\partial D)\to(X^*,X)$ is called \emph{minimal} if $\text{Comp}(D)$ 
is minimal in the lexicographical order among disc diagrams with the same boundary path as $D$. 
Whenever the boundary path of a cone-cell $C$ in $D$ has a spur, we can replace $C$ by 
a cone-cell with this spur removed without changing the complexity of $D$. Thus 
we can assume that the boundary path of each cone-cell is immersed.

\begin{exa}
Let $X$ be a wedge of circles labelled by $x_1,\dots, x_n$. Suppose $Y_i$ are 
immersed closed combinatorial paths, i.e.\  $Y_i$ corresponds to a cyclically reduced word 
$r_i$ in alphabet $x_1^{\pm 1},\dots, x_n^{\pm1}$. Then $X^*$ of the cubical presentation 
$\langle X, \{Y_1, \dots, Y_m\} \rangle$ is the standard presentation complex associated 
to the group presentation $\langle x_1, \dots, x_n | r_1, \dots ,r_m\rangle$.
\end{exa}

\subsection{Pseudorectangles and ladders}\label{cubicalladdersection}
\begin{defi}\label{rect}A \emph{rectangle} is a squared disc diagram isometric to 
$I_n\times I_m$ for some  natural numbers $n,m$. A \emph{pseudorectangle} is 
a square disc diagram $R$ with $\partial R=e_1\cdots e_n f_1\cdots f_k e_n'\cdots e_1' g_l\cdots g_1$ 
(where $n\geq 1$, $k,l\geq 0$) such that
\begin{itemize}
\item for every $i=1,\dots,n$ we have $\Gamma(e_i)=\Gamma(e_i')$, 
\item for $i\neq j$ we have $\Gamma(e_i)\cap\Gamma(e_j)=\emptyset$, and
\item the concatenation $e_n f_1\cdots f_k e_n'$ is a path in $N\big(\Gamma(e_n)\big)$ and 
$e_1' g_l\cdots g_1 e_1$ is a path in $N\big(\Gamma(e_1)\big)$.
\end{itemize} 
See the left diagram in Figure~\ref{diagramK}. Paths $e_1\cdots e_n$ and $e_n'\cdots e_1'$ 
are called the \emph{(opposite) sides} of a pseudorectangle $R$. 
\end{defi}

Let $D$ be a squared disc diagram with $e_1e_2,e_1'e_2'\subset \partial D$ 
such that for $i=1,2$ we have $\Gamma(e_i)=\Gamma(e_i')$ and 
$\Gamma(e_1)\cap\Gamma(e_2)=\emptyset$. The subdiagram $K$ 
\emph{lying between $\Gamma(e_1)$ and $\Gamma(e_2)$} is the maximal subdiagram 
in the unique connected component of $\overline{D-\Gamma(e_1)\cup \Gamma(e_2)}$ that 
intersects both $N\big(\Gamma(e_1)\big),N\big(\Gamma(e_2)\big)$. See the right diagram 
in Figure~\ref{diagramK}.
\begin{figure}[h]\centering\includegraphics{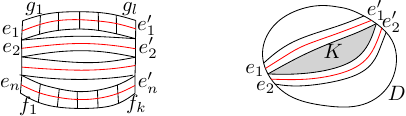}\caption{On the left, a pseudorectangle. On the right, the subdiagram $K$ consists of all squares lying between $\Gamma(e_1)$ and $\Gamma(e_2)$ in $D$ .}\label{diagramK}\end{figure}

\begin{lem}\label{rectangle} Let $D\to X$ be a minimal disc diagram in a non-positively curved 
cube complex $X$. Suppose $D$ is a pseudorectangle, as in Definition~\ref{rect}. Then $k=l$ 
and all squares lying between hyperplanes $\Gamma(e_i)$ (i=1,\dots, n) can be pushed upward, 
i.e.\ there exists a disc diagram $D'$ obtained from $D$ by a sequence of hexagon moves such that 
one of $\Gamma(e_1)$-components of $D'$ is rectangle with sides $e_1\cdots e_n$ and $e_n'\dots e_1'$. 
See Figure~\ref{rectanglelem}.
\begin{figure}[h]\centering\includegraphics{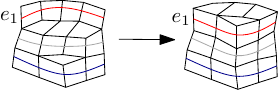}\caption{Pushing squares upward.}\label{rectanglelem}\end{figure}
 \end{lem}
 
\begin{proof}It is immediate that it suffices to prove this lemma for $n=2$. Denote by $m$ 
the number of squares in subdiagram lying between $\Gamma(e_1)$ and $\Gamma(e_2)$. 
We will construct a sequence of diagrams $D=D_m,\dots, D_0$ with the following properties:
\begin{itemize}
\item $D_i$ is obtained from $D_{i+1}$ by a single hexagon move,
\item the subdiagram $K_i$ of $D_i$ lying between $\Gamma(e_1)$ and $\Gamma(e_2)$ has exactly $i$ squares. \end{itemize} 
By definition of $m$ the diagram $D_m$ satisfy the second property. Suppose we have already 
defined $D_m,\dots, D_{i+1}$ (where $i=0,\dots, m-1$), let us define $D_i$. By Theorem~\ref{cubes} 
the diagram $K_{i+1}$ has at least three corners and/or spurs. Denote by $v_{i+1}$ one, 
that is distinct from $e_1\cap e_2$ and $e_1'\cap e_2'$. If $v_{i+1}$ was a spur, there would be 
two squares with two consecutive common edges, which is impossible by the minimality of $D$. 
See Figure \ref{spurinK_i}.
\begin{figure}[h]\centering\includegraphics{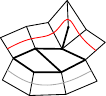}\caption{A spur in $K_i$ violates the minimality of $D$.}\label{spurinK_i}\end{figure}
 Thus $v_{i+1}$ is a corner. Denote by $C$ the square in $K_{i+1}$ containing $v_{i+1}$. 
 There are two more squares in $D_{i+1}$ containing $v_{i+1}$, they are contained in the carrier 
 of exactly one of hyperplanes $\Gamma(e_1), \Gamma(e_2)$. We perform a hexagon move at $v_{i+1}$. 
 We set $D_i$ to be the resulting diagram and we denote by $C_i$ the square opposite to $C$. 
 See Figure~\ref{pushing}. 
\begin{figure}[h]\centering\includegraphics{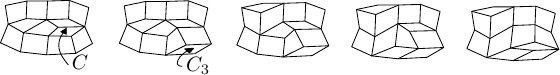}\caption{The sequence of diagrams $D=D_4
,\dots, D_0$.}\label{pushing}\end{figure}
In $D_0$ the diagram $N\big(\Gamma(e_1)\big)\cup N\big(\Gamma(e_2)\big)$ is a rectangle, 
i.e.\ $K_0$ is a path graph, because otherwise there would be three spurs in $K_0$, 
which would contradict the minimality of $D_0$. 
Note that $C_j$ remains in all $D_t$ for $t\leq j$, in particular in $D_0$. Let $n_1,\dots, n_h$ 
be a subsequence of $1,\dots, m$ consisting of exactly those numbers for which $C_{n_j}$ 
was obtained by pushing a square downward, i.e. they appeared in steps where the hexagon move 
was applied to squares such that two of them are contained in $N(\Gamma(e_2))$. In other words, 
$C_{n_j}$ intersects $N(\Gamma(e_2))$ in $D_{n_j}$ \Big(and what follows, $C_{n_j}$ does not intersect $N\big(\Gamma(e_1)\big)$ in $D_{n_j}$ and what follows also in $D_t$ for $t<n_j$\Big). 
Now squares pushed downward, will be pushed ``back''. We define a sequence of disc diagrams 
$D_0=D_0',\dots, D_h'$ in the following way: the diagram $D_{j+1}'$ is obtained from $D_j'$ 
by pushing square $C_{n_j}$ upward, i.e.\ we first apply a hexagon move to $C_{n_j}$ and 
two uniquely determined squares in $N\big(\Gamma(e_2)\big)$ which meet $C_{n_j}$ and 
then we apply a  second hexagon move to the square $\widehat{C}_{n_j}$ opposite to $C_{n_j}$ 
and two uniquely determined squares in $N\big(\Gamma(e_1)\big)$ which meet $\widehat{C}_{n_j}$. 
See Figure~\ref{pushingupward}.
\begin{figure}[h]\centering\includegraphics{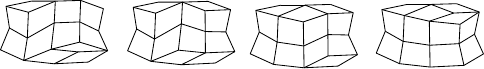}\caption{Pushing squares ``back''. The sequence of diagrams $D_0',\dots, D_h'$.}\label{pushingupward}\end{figure}
 Note that the first hexagon move is the ``inverse'' of the hexagon moves performed 
 in the definition of $D_{n_j}$ and can be performed because all other hexagon moves 
 performed in first part (i.e. those used in definition of $D_i$ for $i=n_j$ $j=1,\dots, h$) 
 leave unchanged $N(\Gamma(e_2))$ and all squares already pushed downward. 
 In each $D_j'$ the subdiagram $N\big(\Gamma(e_1)\big)\cup N\big(\Gamma(e_2)\big)$ 
 is a rectangle and so the second hexagon move is also well defined. 
 We set $D'=D_h'$ and we are done.

\end{proof}

\begin{defi} A \emph{ladder} is a minimal disc diagram $(L,\partial L)\to (X^*,X)$ in 
a presentation complex $X^*$ consisting of a sequence of cone-cells and/or vertices 
$C_1,\dots, C_n$ ($n\geq2$) and square complexes joining them in the following sense:
\begin{itemize}\item if $n=2$ one of the following holds:
\begin{enumerate}[(1)]
\item $C_1$ and $C_2$ are cone-cells glued along a vertex $v$, i.e.\  $L=C_1\cup_v C_2$, or 
\item $C_1$ and $C_2$ are joined by a single edge $e$ where $e\cap C_1, e\cap C_2$ 
are two vertices of $e$, or
\item the diagram consists of a pseudorectangle $R$ and cone cells $C_1,C_2$ 
each attached to one side of $R$,
\end{enumerate}

  \begin{figure}[h]\centering\includegraphics{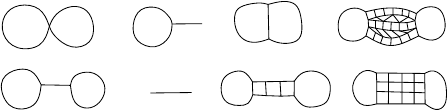}\caption{Ladders for $n=2$.}\end{figure}

\item if $n\geq 3$, then for each $1<i<n$ there are exactly two connected components $L'$ and $L''$ of $L-C_i$ and subdiagrams $L'\cup C_i, L''\cup C_i\subset L$  are both ladders.\end{itemize}

 \begin{figure}[h]\centering\includegraphics{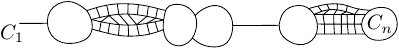}\caption{Example of a ladder.}\end{figure}
The cells $C_1, C_n$ are called \emph{end-cells} of $L$. 

 \end{defi}

\subsection{Pieces}
Given a map $\phi:Y\to X$ an \emph{elevation} of $Y$ to the universal cover $\widetilde X$ of $X$ 
is a map $\widetilde Y\to \widetilde X$ which covers $Y\to X$ such that $\widetilde Y$ is 
the covering space corresponding to $\ker\phi_*$, where $\phi_*:\pi_1(Y)\to\pi_1(X)$ is 
induced by $\phi$. Note that such $\widetilde Y$ is the universal cover of $Y$, 
whenever $\phi_*$ is injective. An \emph{abstract cone-piece} in $Y_i$ of $Y_j$ is the intersection 
$P=\widetilde{Y}_i\cap \widetilde Y_j'$ for some elevations $\widetilde Y_i, \widetilde Y_j' $ of 
$Y_i, Y_j$ to the universal cover $\widetilde X$ of $X$. In the case where $i=j$ we require 
that for the projections $P\to Y_i, P\to Y_j$ there is no automorphism $Y_i\to Y_j$ such that the diagram
\begin{center}
\begin{tikzpicture}
[description/.style={fill=white,inner sep=2pt}]
       \matrix (m) [matrix of math nodes, row sep=2em, column sep=1.5em, text height=1.5ex, text depth=0.25ex]
       { P & Y_i \\
           Y_j& X \\};
       \path[->,font=\scriptsize]
       (m-1-1) edge node[auto] {} (m-1-2)
       (m-1-1) edge node[left] {} (m-2-1)
       (m-1-2) edge node[auto] {} (m-2-2)
       (m-1-2) edge node[auto] {} (m-2-1)
       (m-2-1) edge node[auto] {} (m-2-2);
\end{tikzpicture}
\end{center}
commutes. An \emph{abstract hyperplane-piece} in $Y_i$ is the intersection 
$\widetilde Y_i \cap N(\widetilde A)$, where $\widetilde A$ is a hyperplane in $\widetilde X$ such that 
$\widetilde A\cap \widetilde Y_i =\emptyset$. An \emph{abstract piece} is an abstract cone-piece or 
an abstract hyperplane-piece. A path $\alpha\to Y_i$ is a \emph{piece} (respectively, a \emph{cone-piece}, 
or a \emph{hyperplane-piece}) in $Y_i$, if it lifts to $\widetilde Y_i$ into an abstract piece 
(respectively, an abstract cone-piece, or an abstract hyperplane-piece) in $Y_i$. 
A closed path is \emph{essential} if it is not homotopic to a constant map. 
The cubical presentation $\langle X,\{Y_i\}\rangle$ satisfies \emph{$\mathrm C(p)$-small cancellation condition}, 
if no essential closed path in $Y_i$ can be expressed as a concatenation of fewer than $p$ pieces.

Let $(D,\partial D)\to(X^*,X)$ be a minimal disc diagram. A \emph{$D$-cone-piece} in a cone-cell $C$ 
is a subpath $P$ of $\partial C$ which lies in $C\cap C'$ for some cone-cell $C'\neq C$ in $D$. 
See Figure~\ref{Dpieces}. 
 \begin{figure}[h]\centering\includegraphics{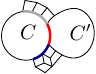}\caption{$D$-cone-piece and $D$-hyperplane-pieces.}\label{Dpieces}\end{figure}
A \emph{$D$-hyperplane-piece} in $C$ is a subpath $P$ of $\partial C$ such that there exists 
a diagram $D'$ obtained from $D$ by a  sequence of hexagon moves and a rectangle 
$I_n\times I_1$ in $D'$ with $P=I_n\times \{1\}$. A \emph{$D$-piece} is a $D$-cone-piece or 
a $D$-hyperplane-piece. See Figure~\ref{Dpieces}. Note that any subpath of a $D$-piece is a $D$-piece.

For every $D$-cone-piece in $C$ there exists a unique maximal $D$-cone-piece containing it, 
but in general this is not true for the $D$-hyperplane-pieces. See Figure~\ref{notunique}. 
 \begin{figure}[h]\centering\includegraphics{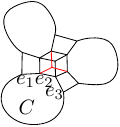}\caption{The maximal $D$-hyperplane-pieces in $C$ are $e_1e_2$ and $e_2e_3$.}\label{notunique}\end{figure}

\begin{lem}Let $\psi: (D,\partial D)\to (X^*,X)$ be a minimal disc diagram in the presentation complex $X^*$. 
Then every $D$-piece in a cone-cell $C$ corresponding to $Y_i$ is mapped under $\psi$ to a piece in $Y_i$.\end{lem}

\begin{proof}First suppose that $P$ is a $D$-cone-piece of $C'$ in $C$, where $C$ corresponds to 
$Y_i$ and $C'$ to $Y_j$. If $i\neq j$, then there is nothing to check. Suppose that $i=j$ and $P$ 
is not mapped to a piece, i.e.\ there is an automorphism $\phi: Y_i\to Y_j$ such that the diagram

\begin{center}
\begin{tikzpicture}
[description/.style={fill=white,inner sep=2pt}]
       \matrix (m) [matrix of math nodes, row sep=2em, column sep=1.5em, text height=1.5ex, text depth=0.25ex]
       { P & Y_i \\
           Y_j& X \\};
       \path[->,font=\scriptsize]
       (m-1-1) edge node[auto] {$ cp$} (m-1-2)
       (m-1-1) edge node[left] {$ c'p' $} (m-2-1)
       (m-1-2) edge node[auto] {} (m-2-2)
       (m-1-2) edge node[auto] {$ \phi $} (m-2-1)
       (m-2-1) edge node[auto] {} (m-2-2);
\end{tikzpicture}
\end{center}
commutes, where $p:P\to \partial C$, $c:\partial C \to Y_i$, $p':P\to \partial C'$, $c':\partial C'\to Y_j$. By the universal property of amalgamated sum, there exists $\mu: C_1 \cup_P C_2\to \text{Cone}(Y_j)$ such that the diagram

\begin{center}
\begin{tikzpicture}
[description/.style={fill=white,inner sep=2pt}]
       \matrix (m) [matrix of math nodes, row sep=1.5em, column sep=1em, text height=1.5ex, text depth=0.25ex]
       { P & C&  \\
          C' & C\cup_P C'&  \\
         &  & \text{Cone}(Y_j) \\};
       \path[->,font=\scriptsize]
       (m-1-1) edge node[auto] {} (m-1-2)
       (m-1-1) edge  node[auto] {} (m-2-1)
       (m-1-2) edge  node[auto] {} (m-2-2)
       (m-2-1) edge node[auto] {} (m-2-2)
       (m-2-2) edge node[auto] {$\mu$} (m-3-3)
       (m-1-2) edge [bend left=20] node[auto] {$\text{Cone}(\phi)\text{Cone}(p)$} (m-3-3)
       (m-2-1) edge [bend right=20] node[left] {$\text{Cone}(p')$} (m-3-3);
\end{tikzpicture}
\end{center}
commutes. We can replace $C$ and $C'$ by a single cone-cell contained in $\text{Cone}(Y_j)$ 
and get a diagram $D'\to X$, with $\partial D'=\partial D$ and $\text{Comp}(D')<_{lex}\text{Comp}(D)$, 
which contradicts the minimality of $D$. 
 
Now suppose $P$ is a $D$-hyperplane-piece in $C$, let $R=I_n\times I_1$ be the rectangle in diagram $D'$ 
obtained from $D$ by a sequence of hexagon moves such that $P=I_n\times \{1\}$. Let $e$ be any edge 
in $\psi(R)$ which has a vertex in $Y_i$, but is not contained in $Y_i$. Let $\widetilde{Y_i}$ be 
some elevation of $Y_i$ and $\tilde e$ a lift of $e$ with a vertex $\tilde v$ in $\widetilde{Y_i}$. 
Then $\tilde e$ is not contained in $\widetilde{Y_i}$, since $e$ was not contained in $Y_i$. 
It suffices to check that $\Gamma(\tilde e)$ does not intersect $\widetilde{Y_i}$. Suppose the contrary 
and denote by $\tilde{e'}$ an edge in $\widetilde{Y_i}$ dual to $\Gamma(\tilde e)$. 
See Figure~\ref{piecestopiecesfig}. Observe that, by Lemma~\ref{convex} $\widetilde{Y_i}$ 
is a convex subcomplex of $\widetilde X$ and by Corollary~\ref{convexcarrier}, also 
$N\big(\Gamma(\tilde{e})\big)$ is convex. The intersection $N\big(\Gamma(\tilde e)\big)\cap \widetilde Y_i$ 
contains $\tilde v$ and $\tilde{e'}$ and since it is convex as an intersection of convex complexes 
it also contains $\tilde e$, which is a contradiction. 
 \begin{figure}[h]\centering\includegraphics{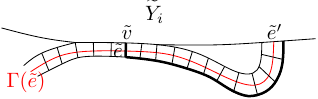}\caption{The thickened path is minimal, so it is contained in $\widetilde Y_i$, since $\widetilde Y_i$ is convex.}\label{piecestopiecesfig}\end{figure}

\end{proof}

Having defined $D$-pieces in a disc diagram $D$, we can adapt the notion of shells. 
A non-disconnecting boundary cell $C$ is a \emph{shell of degree $k$} (or a \emph{$k$-shell} for short) 
if $k$ is the minimal number such that the inner path $P$ of $C$ can be expressed as a concatenation 
of $k$ $D$-pieces. The degree of a cone-cell $C$ in $D$ is denoted by $\deg_D(C)$. 
In the statement of the main result we will also use the notion of corners as defined in 
Section~\ref{seccorners} and spurs as in Section~\ref{shellsandspurs}.

\begin{lem}\label{edgeinpiece} If $C$ is a $k$-shell, then for every $i=1,\dots, k$ there exists 
an edge $e$ in the inner path $P$ of $C$, such that $e\subset P_i$ for any decomposition 
$P=P_1\cdots P_k$ into $D$-pieces.\end{lem}
\begin{proof}
Suppose the contrary and let $i$ be minimal such that the intersection $\bigcap P_i$ over 
all decompositions $P=P_1\cdots P_k$ into $D$-pieces contains no edges. 
Since all $P_i$ are connected subpaths of $P$ it follows that there exist two decompositions 
$P_1\cdots P_k$ and $Q_1\cdots Q_k$ of $P$ into $D$-pieces such that $P_i\cap Q_i$ contains no edges. 
Without loss of generality we can assume that $Q_i$ occurs in $P$ before $P_i$. 
Then \[P_j\cap Q_l \neq \emptyset \text{ for some } j<i \text{ and } l>i.\]
Observe that
\[P_1\cdots P_j \cdot \overline{Q_l-P_j} \cdot Q_{l+1}\cdots Q_k\]
is a decomposition of $P$ into at most $k-1$ $D$-pieces which contradicts the assumption that $C$ is a $k$-shell.
\end{proof}

Every embedded boundary cone-cell $C$ that does not disconnect is a $k$-shell for some $k$, 
since every edge of $\partial C$ internal in $D$ is contained in some $D$-piece. The degree 
of a boundary cone-cell in $D$ is stable under applying hexagon moves in $D$. This follows 
immediately from the definition of $D$-piece.

\begin{exa}
By Lemma~\ref{rectangle}, if $R$ is pseudorectangle with a side $P$ contained in $\partial C$, 
then $P$ is a $D$-hyperplane-piece in $C$, In particular if $L$ is a ladder, then every end-cell $C$ 
is a vertex or \emph{$1$-shell} in $L$ . All other cone-cells are disconnecting, so there are no shells 
of degree $>1$ in $L$.
\end{exa}

\section{The main theorem}

Our goal is to prove the following theorem:
\begin{thm}
\label{cubical}Let $\langle X,\{Y_i\}\rangle$ be a cubical presentation satisfying $\mathrm C(9)$ 
small cancellation condition and let $(D,\partial D)\to (X^*, X)$ be a minimal disc diagram. 
Then one of the following holds:

\begin{itemize}\item $D$ is a single vertex or a single cone-cell,
\item $D$ is a ladder,
\item $D$ has at least three shells of degree $\leq 4$ and/or corner-squares and/or spurs. 
However, if $D$ contains a $4$-shell, there are at least four shells of degree $\leq4$ and/or 
corner-squares and/or spurs.
\end{itemize}
\end{thm}
\noindent We will refer to shells of degree $\leq4$, corner-squares and spurs as \emph{exposed cells}. 
If $D=C$ is a single cone-cell, then $C$ is also called an exposed cell of $D$. 
This theorem for the condition $C'(\frac{1}{12})$ with suitable notion of exposed cone-cells is 
Theorem 9.3 in \cite{raags} or Theorem 3.38 in \cite{hierarchy}. In these papers generalized corners 
are allowed in the place of corners, which gives a slightly weaker statement than here. 
See Lemma~\ref{square} for the definition of generalized corners and note that in \cite{raags} 
they are called \emph{cornsquares}.

\begin{defi}\label{genlad}
A \emph{generalized ladder} is a disc diagram $D$ such that
\begin{itemize}
\item either $D$ is a rectangle $I_n\times I_1$ with $n\geq1$,
\item or $D=R_1\cup_{C_1}L\cup_{C_2}R_2$, where
\begin{itemize}
\item $L$ is a ladder or a single cone-cell or vertex (called the \emph{ladder part} of $D$) 
and $C_1,C_2$ are vertices or edges that in the case where $L$ is a ladder are contained 
in two different end-cells of $L$.
\item for $i=1,2$ the diagram $R_i$ is one of the following:
\begin{itemize}
\item a vertex equal $C_i$, or
\item a rectangle $I_{n_i}\times I_1$ with $n_i\geq 1$ (called an \emph{attached rectangle} of $D$) 
with $C_i=v_i\times I_1$ where $v_i$ is an endpoint of $I_{n_i}$, or 
\item a square (called an \emph{attached square}) and $C_i$ is a vertex of $R_i$.\end{itemize}
\end{itemize}
\end{itemize}
See Figure~\ref{geneladder}.
\begin{figure}[h]\centering\includegraphics{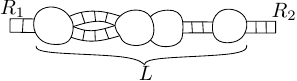}\caption{Generalized ladder.}\label{geneladder}\end{figure}
\end{defi}
A single cell meets the definition of a generalized ladder, while it is not a genuine ladder. Note that all generalized ladders have $\leq2$ exposed cells. In fact, as we will see later, every minimal disc diagram $(D,\partial D)\to(X,X^*)$ with $\leq2$ exposed cells is a generalized ladder. We will inductively prove for $D$ as in Theorem~\ref{cubical} that the following condition (which will be referred to as \emph{condition $(\star)$}) is satisfied:
\begin{itemize}\item $D$ is a generalized ladder, or\item $D$ has at least three exposed cells (shells of degree $\leq 4$ and/or corner-squares and/or spurs). However, if $D$ contains a $4$-shell, there are at least four exposed cells.\end{itemize}

\begin{lem}\label{pushout2}Let $(D,\partial D)\to (X^*, X)$ be a disc diagram such that $D=D_1\cup_C D_2$ where $C$ is a single cell. If both $D_1,D_2$ satisfy $(\star)$, then so does $D$.\end{lem}
\begin{proof}This proof is much like the proof of Lemma~\ref{pushout}. We can assume that $C\subsetneq D_i$ for $i=1,2$, because otherwise there is nothing to prove. If any of the components, say $D_1$, has $\geq3$ exposed cells, then $\geq2$ of them are disjoint from $C$, so they remain exposed in $D$. Together with one exposed cell in $D_2$ (possibly $D_2$ itself in the case where $D_2$ is a single cell) there are $\geq3$ exposed cells in $D$. Similarly, if there is a $4$-shell in $D$, then it is a $4$-shell of one of components $D_1,D_2$, say $D_1$. Thus $D_1$ has $\geq4$ exposed cells, so $\geq3$ remain exposed in $D$ and we have $\geq4$ exposed cells in total in $D$. 

Now suppose that $D_1,D_2$ are both generalized ladders. Then $D$ is a generalized ladder, 
if one of the following holds: 
\begin{itemize}
\item $C$ is a vertex and for $i=1,2$ either the vertex $C$ is contained in an exposed cone-cell/spur 
of $D_i$ (in that case, in particular, $D_i$ has at most one attached rectangle or square) or 
$D_i$ is a single square. See the first diagram in Figure~\ref{glueinggen}.
\item $C$ is a cone-cell and for $i=1,2$ it is exposed in $D_i$.
\item $C$ is an edge and for $i=1,2$ one of the following holds:\begin{itemize} 
\item $C$ is contained in an exposed cone-cell of $D_i$,
\item $C$ is a side of an attached rectangle $R_i$ in $D_i$ opposite to the side contained 
in the ladder part $L_i$ of $D_i$,
\item $D_i$ is a rectangle and $C$ is one of its sides.
\end{itemize}
See the second diagram in Figure~\ref{glueinggen}.

\item $C$ is a square and for $i=1,2$ the square $C$ is a corner-square contained in a rectangle $R_i$, 
that is either an attached rectangle of $D_i$ or $R_i=D_i$. Moreover $R_1\cup_C R_2$ is a rectangle 
whose two opposite sides are contained in cone-cells and/or corner-squares of $D$. 
See the third figure in Figure~\ref{glueinggen}.

\begin{figure}[h]\centering\includegraphics{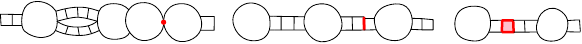}\caption{Gluing generalized ladders along marked cells.}\label{glueinggen}\end{figure}

\end{itemize}
In other cases, there are at least three exposed cells in $D$. Indeed one of the following holds:
\begin{itemize}
\item $C$ is not contained in an exposed cell in one of $D_1,D_2$, say $D_1$. 
In such case there are two exposed cells in $D_1$, which remain exposed in $D$, 
so there are at least three exposed cells in total.
\item $C$ is contained in an attached square $R$ of one of $D_1,D_2$. 
Then $R$ is a corner-square of $D$.
\item for $i=1,2$ the cell $C$ is contained in a rectangle $R_i$, that is either 
an attached rectangle of $D_i$ or $R_i=D_i$, but $R_1\cup_C R_2$ is not a rectangle with 
two opposite sides entirely contained in cone-cells and/or corner-squares of $D_i$.
See Figure~\ref{badglueing}.
\begin{figure}[h]\centering\includegraphics{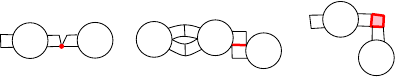}\caption{Gluing generalized ladders along marked cells.}\label{badglueing}\end{figure}
\end{itemize}

\end{proof}

Before the proof of Theorem~\ref{cubical} we define $D$-walls which will play similar role 
in cubical small cancellation complexes as hyperplanes play in cube complexes. 
This is the crucial tool we use in proving Theorem~\ref{cubical}. Next we discuss 
the notion of $\Gamma$-components, similar to one introduced in Section~\ref{sechyperplanes} 
and finally proceed with the proof.

\subsection{$D$-walls}
Throughout this section $(D,\partial D)\to (X^*, X)$ is a minimal disc diagram with no disconnecting cells 
such that all cone-cells embed. The complex $X^*$ is the presentation complex corresponding to 
a presentation $\langle X, \{Y_i\}\rangle$ satisfying $\mathrm C(9)$ condition.

\begin{defi}\label{D-wall} Let $e_0,\dots, e_n$ be a sequence of edges of $D$ and 
$C_1, \dots,C_n$ a sequence of $2$-cells in $D$ (with $e_{i}\neq e_{i+1}$ and $C_i\neq C_{i+1}$) 
such that for all $i=1,\dots, n$ we have
\begin{enumerate}[(1)]
\item either $C_i$ is a square with $e_{i-1}, e_i$ a pair of opposite edges, 
i.e.\  $e_{i-1}\cap e_i=\emptyset$,
\item or $C_i$ is a cone-cell with edges $e_{i-1}, e_i\subset \partial C_i$ such that 
there does not exist a subpath of $\partial C_i$ containing both $e_{i-1}$ and $e_i$, 
that can be expressed as a concatenation of $<5$ $D$-pieces.
\end{enumerate}
Such a pair of sequences $\Gamma=\{(e_i),(C_i)\}$ is called a \emph{$D$-wall}. 
Note that, if $C_n$ is a boundary cone-cell of $D$ and $e_{n-1}$ is an internal or semi-internal edge 
contained in $C_n$, then condition (2) is satisfied for any $e_{n}\subset \partial D\cap\partial C_i$. 
\end{defi}
The $D$-wall $\Gamma$ might be identified with a path graph locally embedded (not combinatorially) 
in $D$ in the following way:\begin{itemize}\item
the vertices of $\Gamma$ correspond to edges $e_0,\dots, e_n$ and they are mapped to the midpoints 
of corresponding edges
\item the edges of $\Gamma$ correspond to cells $C_1, \dots, C_n$ and each edge of $\Gamma$ 
is mapped to a midcube in the square, or respectively to the union of two intervals joining cone-point 
with midpoints of the appropriate edges in the cone-cell. See Figure~\ref{dwall}.
\end{itemize}
 \begin{figure}[h]\centering\includegraphics{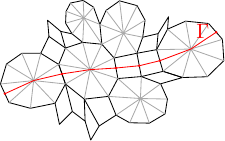}\caption{Example of a $D$-wall $\Gamma$.}\label{dwall}\end{figure}

The $D$-wall $\{(e_{k-1},\dots, e_l),(C_k,\dots, C_l)\}$ 
for some $1\leq k<l\leq n$ is called a \emph{sub-$D$-wall} 
of $\Gamma$. The $D$-wall $\Gamma=\{(e_0,\dots, e_n),
(C_1,\dots, C_n)\}$ with $n\geq 1$ is called 
\begin{itemize}
\item\emph{maximal} if $e_n\subset\partial D$ or $e_n=e_k$ for some $k<n$,
\item \emph{bimaximal} if both $\{(e_0,\dots, e_n),(C_1,\dots, C_n)\}$ 
and $\{(e_n,\dots, e_0),(C_n,\dots, C_1)\}$ are maximal. 
\end{itemize}
Let $e$ be an edge in $D$ and $K\subset D$ a subcomplex, we say that 
\begin{itemize}
\item $\Gamma$ is \emph{dual} to $e$, if there exists $k$ such that $e_k=e$ 
(then we also say that $e$ is dual to $\Gamma$),
\item $\Gamma$ \emph{starts} at $e$ (respectively, in $K$), if $e_0=e$ 
(respectively, if $e_0\subset K$),
\item $\Gamma$ \emph{terminates} at $e$ (respectively, in $K$) if $e_n=e$ 
(respectively, if $e_n\subset K$).\end{itemize}
Let $\Gamma'=\{(e_0',\dots, e_m'), (C_1',\dots, C_m')\}$. 
We say that $\Gamma$ and $\Gamma'$ \emph{intersect} if $C_i=C_j'$ for some $i,j$.

\begin{lem}Let $e$ be an edge in $D$. There exists a bimaximal $D$-wall dual to $e$.\end{lem}
\begin{proof}
Since $D$ is compact, there are finitely many edges in $D$, so it suffices to prove that 
for every cone-cell $C$ and edge $e\subset\partial C$ there exists $e'\subset\partial C$ such that 
$e,e'$ satisfy Condition $(2)$ from Definition~\ref{D-wall}. Indeed, we construct 
a bimaximal $D$-wall step by step until it terminates in $\partial D$ or itself. If $e\subset \partial D$ 
we set $e'$ to any other edge in $\partial C$. If $e\not\subset \partial D$, but $C\cap\partial D$ 
contains some edges, we set $e'$ to any boundary edge in $C$. Assume that $C$ is a cone-cell 
with at most one boundary vertex in $\partial C$. Suppose that for every edge $e'$ in $\partial C$ 
there is a path in $\partial C$ containing both $e,e'$, that is a concatenation of $\leq4$ $D$-pieces. 
Then there exists a pair of such paths that covers whole $\partial C$, thus $\partial C$ can be expressed 
as a concatenation of $8$ $D$-pieces, which is a contradiction with $\mathrm C(9)$ condition. 
Thus there exists a required $e'$.

\end{proof}

Note that in general a maximal $D$-wall dual to an edge $e$ is not unique. 
The notions of hyperplanes and maximal $D$-walls starting at boundary edges coincide if 
$D\to X^*$  factors as $D\to X\to X^*$, i.e.\  if $D$ consists of squares only. 

A $D$-wall $\Gamma$ is \emph{collaring} if $\Gamma$ is not dual to any internal edge in $D$. 
We say that $D$  is \emph{collared}, if for every semi-internal edge $e$ every $D$-wall $\Gamma$ 
dual to $e$ is collaring. We say that $D$ is \emph{collared by a collection of $D$-walls 
$\Gamma_1,\dots, \Gamma_n$}, if $D$ is collared and every semi-internal edge $e$ is dual 
to $\Gamma_i$ for some $i$.

 \begin{figure}[h]\centering\includegraphics{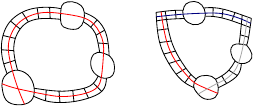}\caption{Collared disc diagrams.}\label{collared2}\end{figure}

\begin{rem}\label{collaringvalence}If $D$ is collared then for every boundary cone-cell $C$ with 
the inner path $e_1\cdots e_n$, paths $e_2\cdots e_n$ and $e_1\cdots e_{n-1}$ both can be expressed 
as concatenations of $<5$ $D$-pieces. In particular, $\deg_{D}C\leq5$.\end{rem}
\begin{proof}Indeed, otherwise the $D$-wall $\{(e_1,e_{n-1}),(C)\}$ or $\{(e_2,e_n),(C)\}$ 
would contradict the assumption that $D$ is collared.\end{proof}

\begin{lem}\label{collaredD}The diagram $D$ with at least one exposed cell is collared if and only if 
all $D$-walls starting in semi-internal edges of exposed cells are collaring.\end{lem}

\begin{proof}
The implication from left to right is trivial. For the other implication, assume that all $D$-walls 
starting in semi-internal edges of exposed cells are collaring. Observe that for every semi-internel edge $e$ 
contained in non-exposed cone-cell $C$ there is an edge $e'$ which is not in $\partial D$ such that 
$\{(e,e'),(C)\}$ is a $D$-wall. Thus for every semi-internal edge $e$ of exposed cell there exists a $D$-wall 
starting at $e$ which also terminates at a semi-internal edge of an exposed cell. It follows that there exists 
a unique collection $\mathcal G$ of $D$-walls with both endpoints in semi-internal edges of exposed cells, 
such that every semi-internal edge in $D$ is dual to an element of $\mathcal G$.

If $D$ is not collared, there are a semi-internal edge $e$, an internal edge $e'$ and a $2$-cell $C$ 
such that $\Gamma=\{(e,e'),(C)\}$ is a $D$-wall. But $e$ is also dual to some $\Gamma'\in\mathcal G$, 
so taking the suitable sub-$D$-wall of $\Gamma'$ and composing it with $\Gamma$, we obtain a 
non collaring $D$-wall starting in a semi-internal edge of an exposed cell of $D$. This is a contradiction.

\end{proof}

\begin{rem}\label{collaringwalls=exposedcells}
Let $D$ be a collared disc diagram with $n\geq1$ exposed cells. We have $|\mathcal G|=n$ 
where $\mathcal G$ is the collection of $D$-walls from the proof of Lemma~\ref{collaredD}. 
The diagram $D$ is collared by $\mathcal G$. For any collection $\mathcal H$ such that $D$ 
is collared by $\mathcal H$ we have $|\mathcal H|\geq n$.

\end{rem}

We define the \emph{$D$-carrier} $N(\Gamma)$ of a $D$-wall $\Gamma=\{(e_i),(C_i)\}$ as follows
\[N(\Gamma):=\bigslant{\coprod\limits_{j=1}^{n}C_i}{\sim},\]
where $C_i, C_{i+1}$ are glued along the maximal $D$-piece of $C_i$ in $C_{i+1}$ containing $e_{i+1}$. 
It is immediate that $N(\Gamma)$ is a ladder. There is a natural combinatorial immersion 
$\iota: N(\Gamma)\to D$ whose image is the minimal subcomplex of $D$ that contains $\Gamma$ 
regarded as an immersed path graph. Whenever $\iota$ is an embedding we write $N(\Gamma)$ 
for $\iota\big(N(\Gamma)\big)$.

Let $\Gamma$ be a bimaximal $D$-wall with embedded $D$-carrier. Observe that $D-\Gamma$ 
has exactly two connected components, denote one of them by $K$. The \emph{$\Gamma$-component} 
corresponding to $K$ is defined as $K\cup N(\Gamma)$. It is the minimal subdiagram of $D$ containing 
$N(\Gamma)$ and $K$.

\begin{lem}\label{laddersomething}Let $C$ be a cone-cell which is an end-cell of a ladder $L$ contained in $D$. 
Denote by $P$ the inner path of $C$ in $L$. For any maximal $D$-piece $Q$ in $C$, which has a common edge 
with $P$, we have $P\subset Q$. See the left diagram in Figure~\ref{ladderinside}.
\begin{figure}[h]\centering\includegraphics{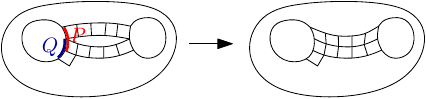}\caption{If $Q$ is maximal $D$-piece, then it contains whole $P$.}\label{ladderinside} \end{figure}
\end{lem}
\begin{proof} We may assume that there is a rectangle in $D$ with side $Q$. Let $R$ be a pseudorectangle 
from the definition of ladder with side $P$. By Lemma~\ref{rectangle} we may push all squares lying between 
hyperplanes dual to edges in $P$ upward, so we get a diagram containing a rectangle with side $P\cup Q$. 
See Figure~\ref{ladderinside}.

\end{proof}

\begin{lem}\label{laddercomponent}Let $\Gamma$ be a bimaximal $D$-wall with embedded carrier and 
denote by $D', D''$ $\Gamma$-components of $D$. If $D'$ is a ladder (see Figure~\ref{ladderlemma}.), 
then $\deg_{D}(C)\leq\deg_{D''}(C)$ for every boundary cone-cell $C$ in $D''$. In particular, 
all exposed cells of $D''$ are also exposed in $D$.\end{lem}

 \begin{figure}[h]\centering\includegraphics{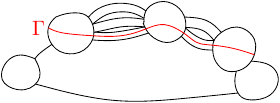}\caption{$D$-wall $\Gamma$.}\label{ladderlemma} \end{figure} 

\begin{proof}Let $C_1,\cdots, C_n$ denote the cone-cells in the ladder $D'$. By Lemma~\ref{laddersomething} 
we have $\deg_{D}(C)=\deg_{D''}(C)$. For any other cone-cell $C\neq C_i$ in $D''$ we have 
$\deg_{D}(C)\leq\deg_{D''}(C)$, because every hexagon move in $D''$ can be also performed 
in the bigger diagram $D$ which contains $D''$ and what follows every $D''$-piece of $C$ is a $D$-piece. 
It remains to verify that there are no exposed cells in $D''$, which are internal in $D$. 
All boundary cells of $D''$, which are internal in $D$ lie in $N(\Gamma)$, so they are not exposed in $D''$, 
by the definition of $D$-wall.\end{proof}

\begin{rem}\label{ladderremark} By the definition of a ladder, if both $\Gamma$-components in Lemma~\ref{laddercomponent} are ladders, then so is $D$.\end{rem}

\subsection{Preliminaries}
Let us now prove the following lemmas useful in the proof of Theorem~\ref{cubical}.

\begin{lem}[Lemma 3.6 in \cite{raags}]\label{square}
Let $D\to X$ be a minimal disc diagram in a non-positively curved cube complex $X$ and 
let $e,e'\subset\partial D$ be a pair of adjacent edges such that $\Gamma(e)$ and $\Gamma(e')$ 
intersect in a square $S$ in $D$. Suppose that $S$ is the only square of the intersection of 
$\Gamma(e),\Gamma(e')$ in $D$ and that $\Gamma(e),\Gamma(e')$ are collaring. 
See Figure~\ref{squarelemma1b}. Such square $S$ together with edges $e,e'$ will be referred to 
as \emph{generalized corner with edges $e,e'$}.

\begin{figure}[h]\centering\includegraphics{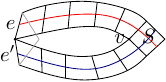}\caption{Generalized corner.}\label{squarelemma1b}\end{figure}
There exists a diagram $D_0\to X$ obtained from $D$ by a sequence of hexagon moves such that 
there is a square $S'$ in $D_0$ with $e'e$ a subpath of $\partial S'$.

\end{lem}
\begin{proof}

Denote by $v$ the unique vertex in $S$ which is internal in $D$. First suppose that 
$\overline{D-S}=I_2\times I_n$ for some $n\geq0$, i.e the internal subdiagram $\text{Int}_D$ is a path graph. 
Denote by $v_1,\dots, v_n$ all consecutive vertices of $\text{Int}_D$ with $v_n=v$. Set $D_n=D$ and 
define $D_{k-1}$ as a diagram obtained from $D_k$ by a hexagon move applied to squares containing 
vertex $v_k$, see Figure~\ref{diagramDk}.
\begin{figure}[h]\centering\includegraphics{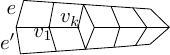}\caption{The diagram $D_k$.}\label{diagramDk}\end{figure}
The diagram $D_0$ contains a square $S$ such that $e e'$ is a subpath of $\partial C$.

In general case by Lemma~\ref{rectangle} applied to $\overline{D-S}$, we know that there exists 
a disc diagram $D'\to X$ obtained from $D$ by a sequence of hexagon moves such that the subdiagram 
lying between $\Gamma(e)$ and $\Gamma(e')$ in $\overline{D-S}$ is a path graph, so there is a subdiagram 
containing $e,e'$ and $S$ which has a form as in the first step. This completes the proof.

\end{proof}

Let us state two corollaries of Lemma~\ref{square}, which are useful in the proof of Theorem~\ref{cubical}. 
We assume that $(D,\partial D)\to (X^*, X)$ is a disc diagram in the presentation complex $X^*$ 
corresponding to a presentation $\langle X, \{Y_i\}\rangle$ which satisfies $\mathrm C(9)$ condition.
The first corollary is an immediate consequence of the assumption that maps $Y_i\to X$ are local isometries:

\begin{cor}\label{squarecor}Suppose $D$ contains a cone-cell $C$ and a generalized corner with edges $e,e'\subset \partial C$. See Figure~\ref{squarecor1}. Then $D$ is not minimal. \begin{figure}[h]\centering\includegraphics{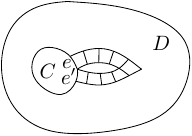}\caption{This diagram is not minimal.}\label{squarecor1}\end{figure}
\end{cor}

\begin{cor}\label{singlepiece}
Suppose that $(D,\partial D)\to (X^*,X)$ is a minimal disc diagram, $C\subset D$ is a cone-cell, 
$P=e_1\cdots e_n$ is a subpath of $\partial C$ and $e$ is an edge which has a common vertex with $e_1$, 
but is not contained in $\partial C$. Suppose that there is a subdiagram $D_P$ of $D$ consisting of squares only 
and containing $P$ and $e$ such that hyperplanes $\Gamma(e)$ and $\Gamma(e_n)$ intersect in $D_P$. 
See the left diagram in Figure~\ref{lemusingsquare}. Then there exists a diagram $D_P'$ 
obtained from $D_P$ by a sequence of hexagon moves such that $P\subset N\big(\Gamma(e)\big)$ 
in $D_P'$, see the right diagram in Figure~\ref{lemusingsquare}.  In particular, $P$ is a $D$-piece.

\begin{figure}[h]\centering\includegraphics{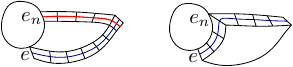}\caption{The diagrams $D_P$ and $D_P'$.}\label{lemusingsquare}\end{figure}\end{cor}
\begin{proof} The proof is by induction on the length of $P$. If $n=1$, i.e.\ $P=e_1$, then 
the assertion follows immediately from Lemma~\ref{square} applied to the generalized corner 
with edges $e,e_1$. Suppose $n>1$, by Corollary~\ref{squarecor} we know that $\Gamma(e_i)$ 
and $\Gamma(e_{i+1})$ do not intersect for any $i=1,\dots, n-1$. It follows that $\Gamma(e_{1})$ 
intersects $\Gamma(e)$, hence by Lemma~\ref{square} there is a disc diagram $D'$ obtained from 
$D_P$ by a sequence of hexagon moves such that there is a square $S$ in $D'$ with $ee_1\subset \partial S$. 
Denote by $e'$ the edge opposite to $e$ in $S$ and set $P'=e_2\cdots e_n$. Note that $\Gamma(e_n)$ 
and $\Gamma(e)=\Gamma(e')$ intersect in $D'$, since $\Gamma(e_n)$ and $\Gamma(e)$ intersect in $D_P$. 
By the induction assumption applied to $P'$, $e'$ and the appropriate diagram there is a diagram $D_P'$ 
obtained from $D'$ by a sequence of hexagon moves leaving $S$ unchanged such that $P'$ is a path 
in $N\big(\Gamma(e')\big)$, by construction  so is $P$.
\end{proof}

\subsection{Proof of the main theorem}

In this section we prove Theorem~\ref{cubical}. The proof is divided into nine steps. 
The first three steps allow us to reduce the problem to diagrams with nontrivial internal subdiagram 
and all cells embedded and not disconnecting. In the fourth step we show that the $D$-carriers 
of $D$-walls embed. In the next two steps we restrict our attention to collared diagrams of two types: 
diagrams with exactly two exposed cells (in that case we intend to verify that they are ladders) and 
diagrams with three exposed cells with a $4$-shell among them (in that case we intend to obtain a contradiction). 
In Step 7 we prove that the internal subdiagrams are squared. Finally, in two last steps we show that there 
are no non-exposed cone-cells in our diagrams and what follows they are ladders in the first case and in the second case we obtain a contradiction.

It is immediate that condition $(\star)$ (formulated after Definition~\ref{genlad}) implies the hypothesis 
of the theorem. We will prove by induction on the number of cells that all minimal disc diagrams satisfy $(\star)$. 
We assume that all disc diagrams having fewer cells than $D$ satisfy $(\star)$ and deduce that so does $D$.

\begin{step}All cone-cells in $D$ are embedded. The intersection of two cone-cells is connected.\end{step}
\begin{proof}First suppose that $C$ is a cone-cell that does not embed. Let $P$ be the minimal subpath 
of $\partial C$ such that its endpoints are mapped to the same point $p$ in $D$. The path $P$ is 
a boundary path of the disc diagram $D'$, which is the closure of a connected component of $D-C$. 
See the left diagram in Figure~\ref{notembeddedcell}.
 \begin{figure}[h]\centering\includegraphics{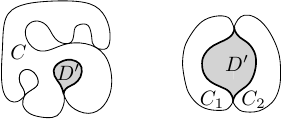}\caption{Cell $C$ is not embedded. 
 Cells $C_1,C_2$ intersect in two points.}\label{notembeddedcell}\end{figure}
There are no spurs in $D'$ by the minimality of $P$. Observe that for any shell $C'$ in $D'$, 
the connected intersection $C'\cap\partial D'$ is a path that can be expressed as a concatenation of 
$\leq2$ $D$-pieces. Thus there are no exposed cone-cells in $D'$ and $D'$ is not a single cone-cell. 
No vertex of $P$, except for $p$ possibly, is a corner of $D'$. Hence, $D'$ does not satisfy $(\star)$, 
which contradicts the induction assumption.
 
Let us now prove the second statement. Suppose that the intersection of cone-cells $C_1,C_2$ is not connected. 
Let $P_1\to\partial C_1, P_2\to\partial C_2$ be minimal paths such that $P_1,P_2$ have common endpoints in $D$. 
Their concatenation is a boundary path of the disc diagram $D'$, which is the closure of a connected component 
of $D-C_1\cup C_2$. See the right diagram in Figure~\ref{notembeddedcell}. Similarly as before, we conclude 
that $D'$ does not contain exposed cone-cells and has no more than two corners, hence $D'$ contradicts 
the induction assumption.
\end{proof}

\begin{step}We may assume that $D$ has no disconnecting cells.\end{step}
\begin{proof}This follows from Lemma~\ref{pushout2}. \end{proof}

\begin{step}We may assume that $\text{Int}_D\neq\emptyset$.\end{step}\begin{proof} Since $D$ has 
no disconnecting cells, by Lemma~\ref{nodisconnect} either $D$ consists of at most two cells, or 
$\text{Int}_D\neq\emptyset$. If $D$ is a single cell, there is nothing to prove. If $D$ consists of two cells, 
then there is a path of length $\geq2$ contained in their intersection, since there are no disconnecting cells. 
If $D$ consists of two cone-cells, then $D$ is a ladder. Otherwise, if $D$ contains a square, $D$ is not minimal.
\end{proof}
To prove $(\star)$ we will verify that
\begin{itemize}
\item either $D$ is a ladder consisting of two cone-cells joined by a pseudorectangle,
\item or there are at least three exposed $2$-cells in $D$.
\end{itemize}

\begin{step} For any $D$-wall $\Gamma$ the $D$-carrier $N(\Gamma)$ embeds in $D$.\end{step}\begin{proof} 

Suppose to the contrary that $\iota:N(\Gamma)\to D$ is not an embedding. Let 
$\Gamma'=\{(e_0,\dots, e_n),(C_1,\dots, C_n)\}$ be a minimal sub-$D$-wall of $\Gamma$ such that 
$N(\Gamma')$ does not embed. Since the intersection of any two cells is connected, we have $n\geq3$. 
Denote by $K$ the component of $D-\iota\big(N(\Gamma')\big)$ which intersects $\partial D$ trivially. 
Let $D_{\Gamma'}$ be the minimal subdiagram containing $\iota\big(N(\Gamma')\big)$ and $K$. 
 \begin{figure}[h]\centering\includegraphics{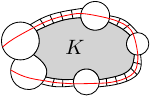}\caption{The diagram $D_{\Gamma'}$.}\label{wallsembedcubical}
\end{figure}
See Figure~\ref{wallsembedcubical}. 
By the minimality of $\Gamma'$ the diagram $D_{\Gamma'}$ is a disc diagram. One of the following holds:
\begin{itemize}
\item the diagram $D_{\Gamma'}$ is a proper subdiagram of $D$. The images of end-cells of $N(\Gamma')$ 
are the only possible exposed cells in $D_{\Gamma'}$. Since they are shells of degree $\geq2$, 
the diagram $D_{\Gamma'}$ cannot be a ladder, so since $D_{\Gamma'}$ has fewer cells than $D$ 
we obtain a contradiction with the induction assumption.
\item we have $D_{\Gamma'}=D$. Choose any boundary cell $C$ distinct from $C_1, C_n$ and consider 
a bimaximal $D$-wall $\Gamma''$ dual to an edge of $C$ that is contained in any piece intersecting $\partial D$. 
If $N(\Gamma'')$ is not embedded, then proceeding as in the first case, we obtain a proper subdiagram 
which contradicts the induction assumption. If $N(\Gamma'')$ is embedded, then one of 
the $\Gamma''$-components $D'$ is a diagram collared by $\Gamma''$ and sub-$D$-wall of $\Gamma$, 
so $D'$ has $\leq2$ exposed cells. The cell $C$ is either a corner-square or a shell of degree $\geq2$ in $D'$, 
so $D'$ is not a ladder, a contradiction.
\end{itemize}
\end{proof}

\begin{step}The diagram $D$ has at least two exposed cells.\end{step}\begin{proof}
Suppose $D$ has $\leq1$ exposed cells. Let $C$ be any non-exposed boundary cell in $D$ and 
let $\Gamma$ be a bimaximal $D$-wall dual to an edge of $C$ that is not contained in any piece 
intersecting $\partial D$. There is at least one $\Gamma$-component $D'$ that has $\leq2$ exposed cells. 
Since $C$ is a corner-square or a shell of degree $\geq2$ in $D'$, we conclude that $D'$ is not a ladder. 
The diagram $D'$ has less cells than $D$, so we obtain a contradiction with the induction assumption.
 \end{proof}
 
From now on we assume that \begin{enumerate}[(A)]\item 
the diagram $D$ has exactly $2$ exposed cells $C_1, C_2$, or
\item the diagram $D$ has three exposed cells $C_1, C_2, C_3$ and $C_3$ is a $4$-shell.\end{enumerate}
To complete the proof of the theorem we will verify that $D$ is a ladder in Case (A) and we will obtain 
a contradiction in Case (B).

\begin{step}\label{stepcollared}The diagram $D$ is collared and both $C_1, C_2$ have inner paths 
of length $2$ or $D$ is a ladder.\end{step}\begin{proof}
By Lemma~\ref{collaredD}, to verify that $D$ is collared and the internal paths of $C_1,C_2$ have length $2$ 
it suffices to check that all $D$-walls starting in $C_1,C_2$ \big(or in a semi-internal edge of $C_3$ in Case (B)\big) 
are collaring. Suppose that $\Gamma$ is a non collaring $D$-wall starting in one of $C_1,C_2$ 
\big(or in a semi-internal edge of $C_3$ in Case (B)\big). Let us consider Cases (A) and (B) separately:
\begin{enumerate}[(A)]
\item One of $\Gamma$-components has $\leq2$ exposed cells, so by the induction assumption is a ladder. 
By Lemma~\ref{laddercomponent} the other $\Gamma$-component has $\leq2$ exposed cells, 
so by the induction assumption it is also a ladder (this happens only if $\Gamma$ starts and terminates 
in $C_1$ and $C_2$). By Remark~\ref{ladderremark} the diagram $D$ is a ladder.
\item If there is a $\Gamma$-component with $\leq2$ exposed cells, then by the induction assumption 
it is a ladder and by Lemma~\ref{laddercomponent} the other $\Gamma$-component $D''$ has 
either $3$ exposed cells with a shell $C_3$ of degree $4$ among them or $\leq2$ exposed cells and 
at least one non-exposed shell. The second case occurs if $\deg_{D''}(C_3)>\deg_{D}(C_3)=4$, 
i.e.\ some $D$-piece in the inner path of $C_3$ in $D''$ is not a single $D''$-piece. In both cases 
$D''$ contradicts the induction assumption. If none of $\Gamma$-components has $\leq2$ exposed cells, 
then both have three exposed cells and one of them contains a $4$-shell, a contradiction.
\end{enumerate}
It follows that $D$ is collared and $C_1, C_2$ have inner paths of length $2$. 
\end{proof}
By Remark~\ref{collaringwalls=exposedcells}:
\begin{enumerate}[\text{In Case} (A)]
\item there exist $D$-walls $\Gamma, \Gamma'$ such that $D$ is collared by $\Gamma, \Gamma'$.
\item there exist $D$-walls $\Gamma, \Gamma_1', \Gamma'_2$ such that $D$ is collared by $\Gamma,\Gamma'_1,\Gamma_2'$ and $\Gamma'_i$ starts in $C_i$ and terminates in $C_3$. 
See Figure~\ref{ABcollared}.
\end{enumerate}
Set $\Gamma=\{(e_0^{\Gamma},\dots, e_m^{\Gamma}),(C_1^{\Gamma},\dots, C_m^{\Gamma})\}$ 
such that $e_o^{\Gamma}\subset C_1$ and $e_m^{\Gamma}\subset C_2$. Note that $C_1,C_2$ 
can be corner-squares and/or shells.

 \begin{figure}[h]\centering\includegraphics{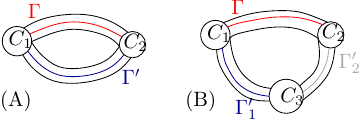}\caption{In Case (A) the diagram is collared 
 by two $D$-walls, in Case (B) the diagram is collared by three $D$-walls.}\label{ABcollared}
 \end{figure}

\begin{step}\label{stepsquares} Internal subdiagram $\text{Int}_D$ of $D$ is a squared diagram.
\end{step}
\begin{proof}
Suppose to the contrary that there is an internal cone-cell in $D$ and denote it by $C$. 
First suppose that there exists a $D$-wall starting in $C$ and terminating in $N(\Gamma)$. 
We will discuss the other case in the very end of this step. Let $e_1\cdots e_n$ be 
the boundary path of $C$ where $e_1$ is chosen so that there is a $D$-wall $\Gamma_{e_1}$ 
starting at $e_1$ which terminates in $C_{i_1}^{\Gamma}$ for minimal $i_1$ (i.e.\ 
closest to $C_1$ in $N(\Gamma)$). Let $e_k$ be the edge in $\partial C$ such that 
there exists a $D$-wall $\Gamma_{e_k}$ starting at $e_k$ which terminates in $C_{i_k}^{\Gamma}$ 
for maximal $i_k$. 
 \begin{figure}[h]\centering\includegraphics{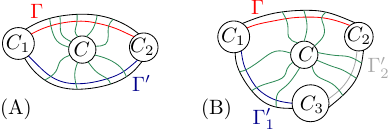}\caption{We suppose there is an internal cone-cell in $D$.}\label{internalconecell}\end{figure}
If $e_1\cdots e_k$ cannot be expressed as a concatenation of $<5$ $D$-pieces, then 
the diagram $D'$ collared by $\Gamma_{e_1}, \Gamma_{e_k}$ and the appropriate sub-$D$-wall 
of $\Gamma$ has $\leq2$ exposed cells and $\deg_{D'}(C)\geq 5$, so $D'$ is not a ladder, 
which is a contradiction with the induction assumption. Hence $e_1\cdots e_k$ can be expressed 
as a concatenation of $\leq4$ $D$-pieces. Observe that for every $l>k$ such that there is a 
$D$-wall $\Gamma_{e_l}$ starting at $e_l$ which intersects $\Gamma_{e_k}$ (respectively, $\Gamma_{e_1}$), 
the diagram collared by the appropriate sub-$D$-walls of $\Gamma_{e_k}$ and $\Gamma_{e_l}$ 
(respectively, $\Gamma_{e_l}$ and $\Gamma_{e_1}$) has $\leq2$ exposed cells. By the induction assumption 
it is a ladder, in particular, by Lemma~\ref{laddersomething}
any maximal $D$-piece containing $e_k$ (respectively $e_1$) contains also $e_l$. It follows that 
the minimal path containing $e_1\cdots e_k$ and every edge $e$ dual to some $D$-wall intersecting 
one of $\Gamma_{e_1},\Gamma_{e_k}$, can be expressed as a concatenation of $\leq4$ pieces. 
Denote by $P$ the maximal subpath of $e_{k+1}\cdots e_n$ such that all $D$-walls starting in $P$ 
intersect none of $\Gamma_{e_1},\Gamma_{e_k}$. By $\mathrm C(9)$ and our last observation $P$ 
cannot be expressed as a concatenation of $\leq5$ $D$-pieces. Denote by $\Gamma_1^P,\Gamma_2^P$ 
$D$-walls starting at two different end-cells of $P$. They do not intersect, because otherwise 
there would be a diagram with only one exposed cell collared by them. Since they do not intersect 
any of $\Gamma_{e_1}, \Gamma_{e_k}$, they both terminate in \begin{enumerate}[(A)]\item
$N(\Gamma')$,  \item$N(\Gamma_1')\cup C_3\cup N(\Gamma_2')$\end{enumerate} in the following way:
\begin{enumerate}[(A)]
\item the diagram $D'$ collared by $\Gamma_1^P,\Gamma_2^P$ and the appropriate sub-$D$-wall 
of $\Gamma'$ has $\leq2$ exposed cells and $\deg_{D'}(C)\geq5$. Thus $D'$ is not a ladder and 
this is a contradiction with the induction assumption.
\item the diagram $D'$ collared by $\Gamma_1^P,\Gamma_2^P$ and the appropriate sub-$D$-walls 
of $\Gamma_1',\Gamma_2'$ has either $\leq 2$ exposed cells and contains a shell of degree $\geq5$, 
or has $\leq3$ exposed cells with a $4$-shell among them. In both cases this is a contradiction with 
the induction assumption.
\end{enumerate}

Now suppose that no $D$-wall starting in $C$ terminates in $N(\Gamma)$. In Case (A) proceed 
exactly as before replacing $\Gamma$ by $\Gamma'$. In Case (B) for $i=1,2$ denote by $\Gamma_i^P$ 
the $D$-wall starting in $C$ and terminating in the closest cell to $C_i$ in $N(\Gamma_1')\cup C_3\cup N(\Gamma_2')$. Similarly as before, the subdiagram collared by $\Gamma_1^P,\Gamma_2^P$ and the appropriate 
sub-$D$-walls of $\Gamma_1',\Gamma_2'$ contradicts the induction assumption. Thus we have shown 
that $D$ has no internal cone-cells.

\end{proof}

\begin{step} In Case (A) the diagram $D$ is a ladder.
\end{step}

\begin{proof}
First, let us show that $D$ has no cone-cells at all 
except for $C_1, C_2$ possibly. Suppose to the contrary 
that $C$ is a non-exposed boundary cone-cell. 
Without loss of generality we can assume that 
$C\subset N(\Gamma)$. By Remark~\ref{collaringvalence} 
we know that $C$ is a $5$-shell. By Lemma~\ref{edgeinpiece} 
there exists an edge $e_2$ which is contained in $P_2$ 
for any decomposition of the inner path $P=P_1\cdots P_5$ 
of $C$ into $D$-pieces. Denote by $\Gamma_{e_2}$ 
a maximal $D$-wall starting at $e_2$. See Figure~\ref{Anoconecell}.
\begin{figure}[h]\centering\includegraphics{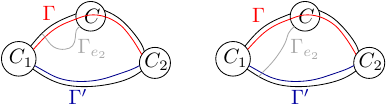}\caption{$\Gamma_{e_2}$ terminates in $N(\Gamma)$ or $N(\Gamma')$.}\label{Anoconecell}\end{figure}
One of the following holds: \begin{itemize}
 
\item If $\Gamma_{e_2}$ terminates in $N(\Gamma)$, 
then the $\Gamma_{e_2}$-component $D'$ collared 
by $\Gamma_{e_2}$ and the appropriate sub-$D$-wall 
of $\Gamma$ has $\leq2$ exposed cells. The inner path 
of $C$ in $D'$ cannot be expressed as a concatenation 
of $2$ (respectively $4$) $D$-pieces, so it cannot be 
expressed as a concatenation of $\geq2$ (respectively $\geq4$) 
$D'$-pieces. Thus $D'$ is not a ladder. This is a contradiction. 

\item If $\Gamma_{e_2}$ terminates in $N(\Gamma')$, 
then the $\Gamma_{e_2}$-component $D''$ such that 
$\deg_{D''}(C)\geq4$ has either three exposed cells with 
a $4$-shell among them (if $\deg{D''}(C)=4$) or 
two exposed cells and a non-exposed cone-cell 
(if $\deg_{D''}(C)>4$). By the induction assumption 
we obtain a contradiction.

\end{itemize}

Hence there are no cone-cells in $D$ except for $C_1,C_2$ possibly. Let us now 
consider different cases depending on what $C_1,C_2$ are:\begin{itemize}
\item if $C_1,C_2$ are both squares, then $D$ is a squared diagram with only two corners 
which is impossible by Theorem~\ref{cubes}.
\item if one of $C_1,C_2$ is a square and the other one is a cone-cell, then $D$ is not minimal 
by Corollary~\ref{squarecor}.
\item if $C_1,C_2$ are both cone-cells, then by definition $D$ is a ladder.
\end{itemize}

\end{proof}

\begin{step} Case (B) is not possible. \end{step}

\begin{proof}
First let us show that there are no cone-cells in $N(\Gamma'_1)\cup N(\Gamma'_2)$. 
Suppose to the contrary that $C$ is a non-exposed cone-cell, say in $N(\Gamma'_1)$, i.e.\ 
there exists $i$ such that $C_i^{\Gamma_1'}=C$ where 
$\Gamma_1'=\{(e_0^{\Gamma_1'},\dots, e_m^{\Gamma_1'}),(C_1^{\Gamma_1'},\dots, C_m^{\Gamma_1'})\}$ 
such that $e_0^{\Gamma_1'}\subset C_1$. Let $P$ be the inner path of $C$ such that 
its first edge is $e_i^{\Gamma_1'}$ and its last edge is $e_{i-1}^{\Gamma_1'}$. 
By Lemma~\ref{edgeinpiece} there is an edge $e$ contained in $P_2$ for any decomposition 
of the inner path $P=P_1\cdots P_5$ of $C$ into $D$-pieces. Denote by $\Gamma_e$ a maximal $D$-wall 
starting at $e$. The degree of $C$ in a $\Gamma_e$-component is $\geq2$ (respectively $\geq4$), 
since the inner path of $C$ in $\Gamma_e$-components cannot be expressed as a concatenation of $2$ 
(respectively $4$) $D$-pieces. If $\Gamma_e$ terminates in $N(\Gamma_1')$ then one 
of $\Gamma_e$-components has $\leq2$ exposed cells and contains a shell of degree $\geq2$, 
so it is not a ladder, a contradiction. 
If $\Gamma_e$ terminates in $N(\Gamma_2')$ then the $\Gamma_e$-component containing $C_3$ 
has $\leq3$ exposed cell and a $4$-shell among them, a contradiction. If $\Gamma_e$ terminates in $N(\Gamma)$, 
then the $\Gamma_e$-component that does not contain $C_3$ have either $3$ exposed cells with 
a $4$-shell $C$ among them or $\leq2$ exposed cell and a non-exposed cone-cell $C$, a contradiction. 
See Figure~\ref{Bpart1}. 
 \begin{figure}[h]\centering\includegraphics{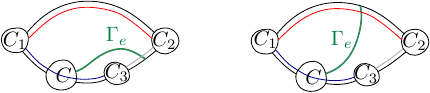}\caption{One of $\Gamma_e$-components has either $3$ exposed cells with a $4$-shell among them, or $\leq2$ exposed cell.}\label{Bpart1}\end{figure}
We have just proved that there are no cone-cells in $N(\Gamma'_1)\cup N(\Gamma_2')$.

Now we will show that the only cone-cells in $D$ are $C_3$ and $C_1,C_2$ possibly. 
It remains to verify that there are no cone-cell in $N(\Gamma)$. Suppose to the contrary that 
$C$ is a non-exposed cone-cell in $N(\Gamma)$, by Remark~\ref{collaringvalence} we know 
that $C$ is a $5$-shell. 

Denote by $e_3$ an edge contained in $P_3$ for every decomposition of the inner path 
$P=P_1\cdots P_5$ of $C$ into pieces. Let $e_2$ ($e_4$ respectively) be the first (the last respectively) 
edge in $P$ that is not contained in $P_1$ ($P_5$ respectively) for any decomposition of $P$ into pieces. 
Note that $e_2$ and $e_3$ ($e_3$ and $e_4$) are not contained in a single piece. Let  $P'$ be 
the minimal subpath of $P$ containing $e_2$ and $e_4$. Observe that every $D$-wall starting in $P'$ 
terminates in $N(\Gamma_1')\cup N(\Gamma_2')$. Otherwise, there would be a subdiagram with 
two exposed cells and a shell of degree $\geq2$ among them, which is impossible by the induction assumption. 
Let us consider three cases:
\begin{itemize}
\item we have $e_3\cap \big(N(\Gamma_1')\cup C_3\cup N(\Gamma_2')\big)=\emptyset$, see Figure~\ref{last1}. 
 \begin{figure}[h]\centering\includegraphics{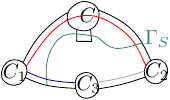}\caption{The edge $e_3$ is not contained in $N(\Gamma_1')\cup C_3\cup N(\Gamma_2')$.}\label{last1}\end{figure}
Denote by $S$ a square in $D$ which contains $e_3$. 
Let $\Gamma_S=\{(e_0^{\Gamma_S},\dots,e_m^{\Gamma_S}),(C_1^{\Gamma_S}, \dots, C_m^{\Gamma_S})\}$ 
be a bimaximal $D$-wall intersecting $S$ but not $e_3$, i.e.\ $S=C_i^{\Gamma_S}$ for 
some $i$ and $e_{i-1}^{\Gamma_S}, e_i^{\Gamma_S}\neq e_3$. See Figure~\ref{last1}. 
Observe that no endpoint of $\Gamma_S$ lies in $N(\Gamma)$, because otherwise by 
Corollary~\ref{singlepiece} the minimal subpath of $\partial C$ containing $e_3$ and one of 
$e_2,e_4$ would be a $D$-piece, but this is not the case. Thus both endpoints of $\Gamma_S$ 
lie in $N(\Gamma_1')\cup N(\Gamma_2')$. The diagram collared by $\Gamma_S$ and 
the appropriate sub-$D$-walls of $\Gamma_1'$ and $\Gamma_2'$ has either three exposed cells 
with a $4$-shell $C_3$ among them, or only two exposed cells and at least one of them is a corner-square. 
In both cases we obtain a contradiction.

\item we have $e_3\subset N(\Gamma_1')\cup N(\Gamma_2')$, see Figure~\ref{last2}. 
Without loss of generality, we may assume that $e_3\subset N(\Gamma_1')$. 
Let $e$ be the one of two edge intersecting $e_3$ and dual to $\Gamma_1'$ that is closer 
to $C_3$ in $N(\Gamma_1')$. Let $\Gamma_{e_2}$ be a $D$-wall dual to $e_2$. Since 
$\Gamma_{e_2}$ terminates in $N(\Gamma_1')$, by Corollary~\ref{singlepiece} applied to 
the minimal subpath of $\partial C$ containing $e_2$ and $e_3$ and to the edge $e$ we conclude 
that $e_2$ and $e_3$  are contained in a single $D$-piece, a contradiction. See Figure~\ref{last2}.
 \begin{figure}[h]\centering\includegraphics{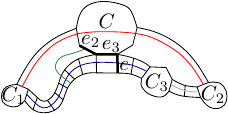}\caption{The edge $e_3$ is contained in $N(\Gamma'_1)$.}\label{last2}\end{figure}

\item we have $e_3\subset C_3$, see Figure~\ref{last3}. 
Let $Q=Q_1\cdots Q_4$ be some decomposition into $D$-pieces of the inner path $Q$ of $C_3$, 
where $\Gamma_1'$ is dual to an edge in $Q_1$. The edge $e_3$ is contained in one of $Q_2$, $Q_3$, 
without loss of generality we can assume that it is in $Q_2$. Note that $Q_3$ is a $D$-hyperplane-piece. 
Denote by $Q'',Q'$ paths such that $Q=Q'' e_3 Q'$, note that $Q_3Q_4\subset Q'$. Let $e$ be the edge 
that occurs right after $e_3$ in $P$. Since $D$-wall starting at $e$ terminates in $N(\Gamma'_2)$, 
by Corollary~\ref{singlepiece} applied to the path $Q'$ and edge $e$, we conclude that $Q_3 Q_4$ is 
a single $D$-piece, which is a contradiction. See Figure~\ref{last3}.
\begin{figure}[h]\centering
\includegraphics{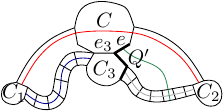}\caption{The edge $e_3$ is contained in $C_3$. }\label{last3}\end{figure}
\end{itemize} 
Thus, there are no cone-cells in $D$ except for $C_3$ and $C_1, C_2$ possibly. 

Finally, we show that such $D$ cannot exist. 
Let $Q$ be the inner path of $C$ with $\Gamma_1'$ dual the first edge $e_1$ of $Q$. Denote by 
$e_2$ ($e_3$ respectively) the first (the last respectively) edge in $Q$ that is not contained 
in $Q_1$ (respectively $Q_4$) for any decomposition $Q=Q_1\cdots Q_4$ into pieces. 
At most one of $e_2,e_3$ is contained in $N(\Gamma)$. Without loss of generality, assume 
that $e_2\not\subset N(\Gamma)$. Denote by $S$ the square in $D$ containing $e_2$ and 
by $e$ the one of two edges of $S$ dual to $\Gamma_S$ that intersect $Q$ farther from $e_1$. 
See Figure~\ref{lastlast}. \begin{figure}[h]\centering
 \includegraphics{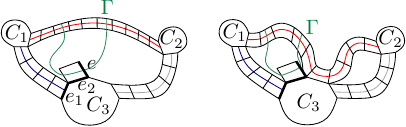}\caption{Corollary~\ref{singlepiece} is applied to the thickened path.}\label{lastlast}\end{figure}
By Corollary~\ref{singlepiece} applied to the minimal subpath of $Q$ containing $e_1$ and $e_2$ 
and to the edge $e$ we know that $\Gamma_S$ cannot have an endpoint in $N(\Gamma_1')$, 
because otherwise $e_1$ and $e_2$ would lie in the single piece of $C$, which is not the case. 
Similarly, we conclude that $\Gamma_S$ do not have endpoints in $N(\Gamma_2')$. It follows that 
both endpoints of $\Gamma_S$ lie in $N(\Gamma)$. Hence there is a squared diagram collared 
by $\Gamma_S$ and a sub-$D$-wall of $\Gamma$, which has only two corners and consists of squares only, 
which is a contradiction and completes the proof.

\end{proof}

\bibliographystyle{plain}
\bibliography{Sans-titre2}

\begin{thebibliography}{1}

\bibitem{greendlinger}
Martin Greendlinger.
\newblock On {D}ehn's algorithms for the conjugacy and word problems, with
  applications.
\newblock {\em Comm. Pure Appl. Math.}, 13:641--677, 1960.

\bibitem{special}
Fr{\'e}d{\'e}ric Haglund and Daniel~T. Wise.
\newblock Special cube complexes.
\newblock {\em Geom. Funct. Anal.}, 17(5):1551--1620, 2008.

\bibitem{vankampen}
Egbert R.~Van Kampen.
\newblock On {S}ome {L}emmas in the {T}heory of {G}roups.
\newblock {\em Amer. J. Math.}, 55(1-4):268--273, 1933.

\bibitem{lyndon}
Roger~C. Lyndon and Paul~E. Schupp.
\newblock {\em Combinatorial group theory}.
\newblock Classics in Mathematics. Springer-Verlag, Berlin, 2001.
\newblock Reprint of the 1977 edition.

\bibitem{fans}
Jonathan~P. McCammond and Daniel~T. Wise.
\newblock Fans and ladders in small cancellation theory.
\newblock {\em Proc. London Math. Soc. (3)}, 84(3):599--644, 2002.

\bibitem{mixed}
Piotr Przytycki and Daniel~T. Wise.
\newblock Mixed manifolds are virtually special.
\newblock \url{http://arxiv.org/abs/1205.6742}.

\bibitem{tartakovskii}
Vladimir A. Tartakovskii.
\newblock Solution of the word problem for groups with a k-reduced basis for k>6. (Russian)
\newblock {\em Izvestiya Akad. Nauk SSSR. Ser. Mat.}, 13:483--494, 1949.


\bibitem{hierarchy}
Daniel~T. Wise.
\newblock The structure of groups with quasiconvex hierarchy.
\newblock \url{https://docs.google.com/open?id=0B45cNx80t5-2T0twUDFxVXRnQnc}.

\bibitem{cubulatingsmall}
Daniel~T. Wise.
\newblock Cubulating small cancellation groups.
\newblock {\em Geom. Funct. Anal.}, 14(1):150--214, 2004.

\bibitem{raags}
Daniel~T. Wise.
\newblock {\em From riches to raags: 3-manifolds, right-angled {A}rtin groups,
  and cubical geometry}, volume 117 of {\em CBMS Regional Conference Series in
  Mathematics}.
\newblock Published for the Conference Board of the Mathematical Sciences,
  Washington, DC, 2012.

\end{thebibliography}

\end{document}